\documentclass{amsart}
\usepackage{amsaddr}
\usepackage{amsmath}%
\usepackage{amsfonts}%
\usepackage{amssymb}%
\usepackage{graphicx}
\usepackage{color}
\usepackage{dsfont}
\usepackage{bbm}
\usepackage{tikz}
\usetikzlibrary{calc,intersections}
\usepackage{url}
\newcommand{\eps}{\varepsilon}

\newcommand{\Fcal}{\mathcal{F}}
\newcommand{\Ical}{\mathcal{I}}

\newcommand{\Lcal}{\mathcal{L}}

\newcommand{\Pcal}{\mathcal{P}}
\newcommand{\Scal}{\mathcal{S}}
\newcommand{\Tcal}{\mathcal{T}}
\newcommand{\Xcal}{\mathcal{X}}
\newcommand{\bbA}{\mathbb{A}}
\newcommand{\bbB}{\mathbb{B}}
\newcommand{\R}{\mathbb{R}}

\newcommand{\bbL}{\mathbb{L}}

\newcommand{\N}{\mathbb{N}}
\newcommand{\bL}{\mathbf{L}}

\newcommand{\bH}{\mathbf{H}}
\newcommand{\bP}{\mathbf{P}}
\newcommand{\bV}{\mathbf{V}}
\newcommand{\bX}{\mathbf{X}}
\newcommand{\bY}{\mathbf{Y}}
\newcommand{\avec}{\mathbf{a}}
\newcommand{\bvec}{\mathbf{b}}
\newcommand{\dvec}{\mathbf{d}}
\newcommand{\evec}{\mathbf{e}}
\newcommand{\fvec}{\mathbf{f}}
\newcommand{\nvec}{\mathbf{n}}

\newcommand{\uvec}{\mathbf{u}}
\newcommand{\vvec}{\mathbf{v}}
\newcommand{\wvec}{\mathbf{w}}
\newcommand{\xvec}{\mathbf{x}}
\newcommand{\hatF}{\hat{F}}
\newcommand{\hatK}{\hat{K}}

\newcommand{\Om}{\Omega}
\newcommand{\ds}{\displaystyle}
\newcommand{\pa}{\partial}
\newcommand{\ul}{\underline}

\newcommand\norm[1]{\left\| #1 \right\|}
\DeclareMathOperator{\grad}{\mathbf{grad}}
\DeclareMathOperator{\Grad}{\mathbf{Grad}}
\DeclareMathOperator{\Deltavec}{\mathbf{\Delta}}

\let\div=\divergence
\theoremstyle{plain}
\newtheorem{theo}{Theorem}
\newtheorem{defi}{Definition}
\newtheorem{lemm}{Lemma}
\newtheorem{prop}{Proposition}
\newtheorem{rema}{Remark}
\numberwithin{equation}{section}
\begin{document}
\title[T-coercivity for Stokes]{T-coercivity for solving Stokes problem with nonconforming finite elements}
\author{Erell Jamelot}
\address{Universit\'e Paris-Saclay, CEA, Service de Thermo-hydraulique et de M\'ecanique des
Fluides, 91191 Gif-sur-Yvette, France}
\email{erell.jamelot@cea.fr}
\date{\today}
\begin{abstract}
We propose to analyse the discretization of the Stokes problem with nonconforming finite elements in light of the T-coercivity (cf. \cite{Ciar12} for Helmholtz-like problems, see \cite{JaCi13}, \cite{CiJK17} and \cite{Gire18} for the neutron diffusion equation). We propose explicit expressions of the stability constants. Finally, we give numerical results illustrating the importance of using divergence-free velocity reconstruction.
\noindent \textbf{Funding.} CEA SIMU/SITHY project
\end{abstract}
\maketitle
\section{Introduction}
The Stokes problem describes the steady state of incompressible Newtonian flows. They are derived from the Navier–Stokes equations \cite{GiRa86}. With regard to numerical analysis, the study of Stokes problem helps to build an appropriate approximation of the Navier–Stokes equations. We consider here a discretization with nonconforming finite elements \cite{CrRa73,FoSo83}. We propose to state the discrete inf-sup condition in light of the T-coercivity (cf. \cite{Ciar12} for Helmholtz-like problems, see \cite{JaCi13}, \cite{CiJK17} and \cite{Gire18} for the neutron diffusion equation), which allows to estimate the discrete error constant. In Section \ref{sec:Tcoer}, we recall the T-coercivity theory as written in \cite{Ciar12}. In Section \ref{sec:Stokes} we apply it to the continuous Stokes Problem. We give details on the triangulation in Section \ref{sec:disc}, and we apply the T-coercivity to the discretization of Stokes problem with nonconforming mixed finite elements in Section \ref{sec:Tcoerh}. In Section \ref{sec:CrRa} (resp. \ref{sec:FoSo}), we precise the proof of the well-posedness in the case of order $1$ (resp. order $2$) nonconforming mixed finite elements. Finally, we give numerical results illustrating the importance of using divergence-free velocity reconstruction.
\section{T-coercivity}\label{sec:Tcoer}
We recall here the T-coercivity theory as written in \cite{Ciar12}. Consider first the variational problem, where $V$ and $W$ are two Hilbert spaces and $f\in V'$:
\begin{equation}
\label{eq:var-pb}
\mbox{Find }u\in V\mbox{ such that }\forall v\in W\mbox{, }a(u,v)=\langle f,v\rangle_V.
\end{equation}
Classically, we know that Problem \eqref{eq:var-pb} is well-posed if $a(\cdot,\cdot)$ satisfies  the stability and the solvability conditions of the so-called Banach–Ne\v{c}as–Babu\v{s}ka (BNB) Theorem (see a.e. \cite[Thm. 25.9]{ErGu21-II}). For some models, one can also prove the well-posedness using the T-coercivity theory (cf. \cite{Ciar12} for Helmholtz-like problems, see \cite{JaCi13}, \cite{CiJK17} and \cite{Gire18} for the neutron diffusion equation).
\begin{defi}\label{def:Tcoer}
Let $V$ and $W$ be two Hilbert spaces and $a(\cdot,\cdot)$ be a continuous and bilinear form over $V\times W$. It is $T$-coercive if
\begin{equation}\label{eq:Tcoer}
\exists T\in\Lcal(V,W)\mbox{, bijective, }\exists\alpha>0\mbox{, }\forall v\in V\mbox{, }|a(v,Tv)|\geq\alpha\|v\|_V^2.
\end{equation}
If in addition $a(\cdot,\cdot)$ is symmetric, it is $T$-coercive if
\begin{equation}\label{eq:Tcoer-sym}
\exists T\in\Lcal(V,V)\mbox{, }\exists\alpha>0\mbox{, }\forall v\in V\mbox{, }|a(v,Tv)|\geq \alpha\|v\|_V^2.
\end{equation}
\end{defi}
\textcolor{blue}{When the bilinear form $a(\cdot,\cdot)$ is symmetric, the requirement that the operator $T$ is bijective can be dropped.} It is proved in \cite{Ciar12} that the T-coercivity  condition is equivalent to the stability and solvability conditions of the BNB Theorem. Whereas the BNB theorem relies on an abstract inf–sup condition, T-coercivity uses explicit inf–sup operators, both at the continuous and discrete levels. 
\begin{theo}{(well-posedness)}\label{thm:WellPosed}
Let $a(\cdot,\cdot)$ be a continuous and bilinear form. Suppose that the form $a(\cdot,\cdot)$ is $T$-coercive. Then Problem \eqref{eq:var-pb} is well-posed.
\end{theo}
\section{Stokes problem}\label{sec:Stokes}
Let $\Om$ be a connected bounded domain of $\R^d$, $d=2,\,3$, with a polygonal $(d=2)$ or Lipschitz polyhedral $(d=3)$ boundary $\pa\Om$. We consider Stokes problem:
\begin{equation}\label{eq:Stokes}
\mbox{Find }(\uvec,p)\mbox{ such that }
\left\{\begin{array}{rcl}
-\nu\Deltavec\uvec+\grad p &=&\fvec,\\
\div\uvec&=&0.
\end{array}\right.
\end{equation}
with Dirichlet boundary conditions for the velocity $\uvec$ and a normalization condition for the pressure $p$:
\[
\uvec=0\mbox{ on }\pa\Om,\quad\int_\Om p=0.
\]
The vector field $\uvec$ represents the velocity of the fluid and the scalar field $p$ represents its pressure divided by the fluid density which is supposed to be constant. Thus, the SI unit of the components of $\uvec$ is $m\cdot s^{-1}$ and the SI unit of $p$ is $m^2\cdot s^{-2}$). The first equation of \eqref{eq:Stokes} corresponds to the momentum balance equation and the second one corresponds to the conservation of the mass. The constant parameter $\nu>0$ is the kinematic viscosity of the fluid, its SI unit is $m^2\cdot s^{-1}$. The vector field $\fvec\in\bH^{-1}(\Om)$ represents a body forces divided by the fluid density, its SI unit is $m\cdot s^{-2}$.

Before stating the variational formulation of Problem \eqref{eq:Stokes}, we provide some definition and reminders. Let us set $\bL^2(\Om)=(L^2(\Om))^d$, $\bH^1_0(\Om)=(H^1_0(\Om))^d$, $\bH^{-1}(\Omega)=(H^{-1}(\Om))^d$ its dual space and $L^2_{zmv}(\Om)=\{q\in L^2(\Om)\,|\,\int_{\Om}q=0\}$. We recall that $\bH(\div;\,\Om)=\{\vvec\in\bL^2(\Om)\,|\,\div\vvec\in\ L^2(\Om)\}$. Let us first recall Poincaré-Steklov inequality:
\begin{equation}\label{eq:Poincare}
\exists C_{PS}>0\,|\,\forall v\in H^1_0(\Om),\quad \|v\|_{L^2(\Om)}\leq C_{PS}\|\grad v\|_{\bL^2(\Om)}.
\end{equation}
The SI unit of $C_{PS}$ is $m$. \\
Thanks to this result, in $H^1_0(\Om)$, the semi-norm is equivalent to the natural norm, so that the scalar product reads $(v,w)_{H^1_0(\Om)}=(\grad v,\grad w)_{\bL^2(\Om)}$ and the norm is $\|v\|_{H^1_0(\Om)}=\|\grad v\|_{\bL^2(\Om)}$. Let $\vvec,\,\wvec\in\bH^1_0(\Om)$. We denote by $(v_i)_{i=1}^d$ (resp. $(w_i)_{i=1}^d$) the components of $\vvec$ (resp. $\wvec$), and we set $\Grad\vvec=(\pa_j v_i)_{i,j=1}^d\in\bbL^2(\Om)$, where $\bbL^2(\Om)=[L^2(\Om)]^{d\times d}$. We have: 
\[
(\Grad\vvec,\Grad\wvec)_{\bbL^2(\Om)}=(\vvec,\wvec)_{\bH^1_0(\Om)}=\ds\sum_{i=1}^d(v_i,w_i)_{H^1_0(\Om)}\]
and:
\[
\|\vvec\|_{\bH^1_0(\Om)}=\ds\left(\sum_{j=1}^d\|v_j\|_{H^1_0(\Om)}^2\right)^{1/2}=\|\Grad\vvec\|_{\bbL^2(\Om)}.
\]
Let us set $\bV=\left\{\vvec\in\bH^1_0(\Om)\,|\,\div\vvec=0\right\}$. The space $\bV$ is a closed subset of $\bH^1_0(\Om)$. We denote by $\bV^\perp$ the orthogonal of $\bV$ in $\bH^1_0(\Om)$. Let $\nu_p>0$ be a kinematic viscosity. We recall that \cite[cor. I.2.4]{GiRa86}:
\begin{prop}\label{prop:diviso}
The operator $\div:\,\bH^1_0(\Om)\rightarrow L^2(\Om)$ is an isomorphism of $\,\bV^{\perp}$ onto $L^2_{zmv}(\Om)$. Let $\nu_p>0$ be a constant kinematic viscosity.
We call $C_{\div}$ the constant such that:
\begin{equation}\label{eq:BAI}
\forall p\in L^2_{zmv}(\Om),\,\exists!\vvec_p\in\bV^{\perp}\,|\,\div\vvec_p=\frac{1}{\nu_p}p\mbox{ and }
\|\vvec_p\|_{\bH^1_0(\Om)}\leq \frac{C_{\div}}{\nu_p}\|p\|_{L^2(\Om)}.
\end{equation}
\end{prop}
Here, the constant $C_{\div}$ has no unit. It depends only on the domain $\Om$. Notice that we have: $C_{\div}=1/\beta(\Om)$ where $\beta(\Om)$ is the inf-sup condition (or Ladyzhenskaya–Babu\v{s}ka–Brezzi condition):
\begin{equation}\label{eq:LBB-CIS}
\beta(\Om)=\inf_{q\in L^2_{zmv}(\Om)\backslash \{0\}}\sup_{\vvec\in\bH^1_0(\Om)\backslash \{0\}}\frac{(q,\div\vvec)_{L^2(\Om)}}{\|q\|_{L^2(\Om)}\,\|\vvec\|_{\bH^1_0(\Om)}}.
\end{equation}
Generally, the value of $\beta(\Om)$ is not known explicitly. In \cite{BCDG16}, Bernardi et al established results on the discrete approximation of $\beta(\Om)$ using conforming finite elements. Recently, Gallistl proposed in \cite{Gall19} a numerical scheme with adaptive meshes for computing approximations to $\beta(\Om)$. In the case of $d=2$, Costabel and Dauge \cite{CoDa15} established the following bound: 
\begin{theo}
Let $\Om\subset\R^2$ be a domain contained in a ball of radius $R$, star-shaped with respect to a concentric ball of radius $\rho$. Then
\begin{equation}\label{eq:betaOm}
\beta(\Om)\geq\frac{\rho}{\sqrt{2\,R}}\left(1+\sqrt{1-\frac{\rho^2}{R^2}}\right)^{-1/2}\geq\frac{\rho}{2\,R}.
\end{equation}
\end{theo}
Let us detail the bound for some remarkable domains. If $\Om$ is a ball, $\beta(\Om)\geq\frac{1}{2}$ and if $\Om$ is a square, $\beta(\Om)\geq\frac{1}{2\,\sqrt{2}}$. Suppose now that $\Om$ is stretched in some direction by a factor $k$, then $\beta(\Om)\geq\frac{1}{2\,k}$. Finally, if $\Om$ is L-shaped (resp. cross-shaped) such that $L=k\,l$, where $L$ is the largest length and $l$ is the smallest length of an edge, then $\beta(\Om)\geq\frac{1}{2\,\sqrt{2}\,k}$ (resp. $\beta(\Om)\geq\frac{1}{4\,k}$).

The variational formulation of Problem \eqref{eq:Stokes} reads:\\
Find $(\uvec,p)\in\bH^1_0(\Om)\times L^2_{zmv}(\Om)$ such that
\begin{equation}\label{eq:Stokes-FV}
\left\{\begin{array}{rcll}
\nu(\uvec,\vvec)_{\bH^1_0(\Om)}-(p,\div\vvec)_{L^2(\Om)}&=&\langle\fvec,\vvec\rangle_{\bH^1_0(\Om)}&\forall\vvec\in\bH^1_0(\Om)~;\\
(q,\div\uvec)_{L^2(\Om)}&=&0&\forall q\in L^2_{zmv}(\Om).
\end{array}\right.
\end{equation}
Classically, one proves that Problem \eqref{eq:Stokes-FV} is well-posed using Poincaré-Steklov inequality \eqref{eq:Poincare} and Prop. \ref{prop:diviso}. Check for instance the proof of \cite[Thm. I.5.1]{GiRa86}.

Let us set $\Xcal=\bH^1_0(\Om)\times L^2_{zmv}(\Om)$ which is a Hilbert space which we endow with the following norm:
\begin{equation}\label{eq:normX}
\|(\vvec,q)\|_{\Xcal}=\|\vvec\|_{\bH^1_0(\Om)}+\nu^{-1}\,\|q\|_{L^2(\Om)}.
\end{equation}
We consider now the following bilinear symmetric and continuous form:
\begin{equation}\label{eq:aform}
\left\{
\begin{array}{rcl}
a_S:\Xcal\times\Xcal&\rightarrow&\R\\
(\uvec',p')\times(\vvec,q)&\mapsto&\nu(\uvec',\vvec)_{\bH^1_0(\Om)}-(p',\div\vvec)_{L^2(\Om)}-(q,\div\uvec')_{L^2(\Om)}
\end{array}\right..
\end{equation}
We can write Problem \eqref{eq:Stokes} in an equivalent way as follows:
\begin{equation}\label{eq:Stokes-FVa}
\mbox{Find }(\uvec,p)\in\Xcal\mbox{ such that}\quad a_S\,\left((\uvec,p),(\vvec,q)\right)=\langle\fvec,\vvec\rangle_{\bH^1_0(\Om)}\quad\forall(\vvec,q)\in\Xcal.
\end{equation}

Let us prove that Problem \eqref{eq:Stokes-FVa} is well-posed using the T-coercivity theory.
\begin{theo}
\label{thm:StokesWellPosed}
Problem \eqref{eq:Stokes-FVa} is well-posed. It admits one and only one solution such that:
\begin{equation}\label{eq:WellPosed}
\forall\fvec\in\bH^{-1}(\Omega),\quad\left\{\begin{array}{rcl}
\|\uvec\|_{\bH^1_0(\Om)}\leq\nu^{-1}\,\|\fvec\|_{\bH^{-1}(\Omega)},\\
\|p\|_{L^2(\Om)}\leq C_{\div}\,\|\fvec\|_{\bH^{-1}(\Omega)}.\\
\end{array}\right.
\end{equation}
\end{theo}
\begin{proof}
We follow here the proof given in \cite{BaCi22,Ciar21}. Let us consider $(\uvec',p')\in\Xcal$ and let us build $(\vvec^\star,q^\star)=T(\uvec',p')\in\Xcal$ satisfying \eqref{eq:Tcoer-sym} (with $V=\Xcal$). We need three main steps.
\begin{itemize}
\item[1.] According to Prop. \ref{prop:diviso}, there exists $\vvec_{p'}\in\bH^1_0(\Om)$ such that: $\ds\div\vvec_{p'}=\nu^{-1}\,p'$ in $\Om$ and
\begin{equation}\label{eq:GiRa}
\|\vvec_{p'}\|_{\bH^1_0(\Om)}^2\leq \left(\frac{C_{\div}}{\nu}\right)^2\,\|p'\|_{L^2(\Om)}^2.
\end{equation}
Let us set $(\vvec^\star,q^\star):=(\gamma\,\uvec'-\vvec_{p'},-\gamma\,p')$, with $\gamma>0$. We obtain:
\begin{equation}\label{eq:coer1}
a_S\,\left(\,(\uvec',p'),\,(\vvec^\star,q^\star)\,\right)=
\nu\,\gamma\,\|\uvec'\|_{\bH^1_0(\Om)}^2+\nu^{-1}\,\|p'\|_{L^2(\Om)}^2-\nu\,(\uvec',\vvec_{p'})_{\bH^1_0(\Omega)}.
\end{equation}
\item[2.] In order to bound the last term of \eqref{eq:coer1}, we use Young inequality and then inequality \eqref{eq:GiRa}:
\begin{equation}\label{eq:GvGvp}
(\uvec',\vvec_{p'})_{\bH^1_0(\Omega)}\leq
\ds\frac{\eta}{2}\|\uvec'\|_{\bH^1_0(\Omega)}^2+\frac{\eta^{-1}}{2}\,\left(\frac{C_{\div}}{\nu}\right)^2\,\|p'\|_{L^2(\Om)}^2.
\end{equation}
\item[3.] 
Using the bound \eqref{eq:GvGvp} in \eqref{eq:coer1} and choosing $\eta=\gamma$, we get:
\[
a_S\,\left(\,(\uvec',p'),\,(\vvec^\star,q^\star)\,\right)\geq\frac{\gamma}{2}\,\nu\,\|\uvec'\|_{\bH^1_0(\Omega)}^2+\nu^{-1}\,\left(1+\frac{\gamma^{-1}}{2}(C_{\div})^2\right)\|p'\|_{L^2(\Om)}^2
\]
Consider now $\gamma=(C_{\div})^2$. \\
Noticing that $\nu\|\uvec\|_{\bH^1_0(\Om)}^2+\nu^{-1}\|p'\|_{L^2(\Om)}^2\geq\ds\frac{\nu}{2}\|(\uvec',p')\|_{\Xcal}^2$, we obtain:
\[
a_S\,\left(\,(\uvec',p'),\,(\vvec^\star,q^\star)\,\right)\geq \frac{\nu}{2}\,C_{\min}\,\|(\uvec',p')\|_{\Xcal}^2\mbox{ where }C_{\min}=\frac{1}{2}\min(\,(C_{\div})^2,1\,).
\]
\end{itemize}
The operator $T$ such that $T(\uvec',p')=(\vvec^\star,q^\star)$ is linear and continuous: \[
\begin{array}{rcl}
\|T(\uvec',p')\|_{\Xcal}&:=&\ds\|\vvec^\star\|_{\bH^1_0(\Om)}+\nu^{-1}\,\|q^\star\|_{L^2(\Omega)}\\
\\
&\leq&\ds\gamma\,\|\uvec'\|_{\bH^1_0(\Om)}+\|\vvec_{p'}\|_{\bH^1_0(\Om)}+\gamma\,\nu^{-1}\,\|p'\|_{L^2(\Omega)},\\
\\
&\leq&\ds
\gamma\,\|\uvec'\|_{\bH^1_0(\Omega)}^2+\,(\,C_{\div}+\gamma)\,\nu^{-1}\,\|p'\|_{L^2(\Omega)}^2,\\
\\
&\leq& C_{\max}\,\|(\uvec',p')\|_{\Xcal},
\end{array}
\]
where $C_{\max}=C_{\div}\,(1+C_{\div})$.\\
\footnote{\textcolor{blue}{Remark that $(\vvec^\star,q^\star)=({\bf{0}},0)\Leftrightarrow(\uvec',p')=({\bf{0}},0)$: the operator $T\in\Lcal(\Xcal,\Xcal)$ is bijective.}} The symmetric and continuous bilinear form $a(\cdot,\cdot)$ is then $T$-coercive and according to Theorem \ref{thm:WellPosed}, Problem \eqref{eq:Stokes-FVa} is well-posed. Let us prove \eqref{eq:WellPosed}. Consider $(\uvec,p)$ the unique solution of Problem \eqref{eq:Stokes-FVa}. Choosing $\vvec=0$, we obtain that $\forall q\in L^2_{zmv}(\Om)$, $(q,\div\uvec)_{L^2(\Om)}=0$, so that $\uvec\in\bV$. Now, choosing $\vvec=\uvec$ and using Cauchy-Schwarz inequality, we have: $\nu\,\|\uvec\|_{\bH^1_0(\Om}^2=\langle\fvec,\uvec\rangle_{\bH^1_0(\Om)}\leq\|\fvec\|_{\bH^{-1}(\Om)}\,\|\uvec\|_{\bH^1_0(\Om)}$, so that: $\|\uvec\|_{\bH^1_0(\Om}\leq\nu^{-1}\,\|\fvec\|_{\bH^{-1}(\Om)}$. Next, we choose in \eqref{eq:Stokes-FVa} $\vvec=\vvec_p\in\bV^{\perp}$, where $\div\vvec_p=-\nu^{-1}\,p$ (see Prop. \ref{prop:diviso}). Noticing that $\uvec\in\bV$ and $\vvec_p\in\bV^\perp$, it holds\footnote{\textcolor{blue}{According to \cite[Cor. I.2.3]{GiRa86}, since $\vvec_p\in\bV^{\perp}$, $\exists p\in L^2_{zmv}(\Om)\,|\,\Delta\vvec_p=\grad q$ in $\bH^{-1}(\Om)$. Integrating by parts twice, we have: $(\uvec,\vvec_p)_{\bH^1_0(\Om)}=-\langle\grad q,\uvec\rangle_{\bH^1_0(\Om)}=(q,\div\uvec)_{L^2(\Om)}=0$.}}: $(\uvec,\vvec_p)_{\bH^1_0(\Om)}=0$. This gives:
\[
\begin{array}{rcl}
-(p,\div\vvec_p)_{L^2(\Om)}&=&\nu^{-1}\,\|p\|_{L^2(\Om)}^2=\langle\fvec,\vvec_p\rangle_{\bH^1_0(\Om)},\\
&\leq&
\|\fvec\|_{\bH^{-1}(\Om)}\,\|\vvec_p\|_{\bH^1_0(\Om)}\leq C_{\div}\,\nu^{-1}\|\fvec\|_{\bH^{-1}(\Om)}\,\|p\|_{L^2(\Om)},
\end{array}
\]
so that: $\|p\|_{L^2(\Om)}\leq C_{\div}\|\fvec\|_{\bH^{-1}(\Om)}$.
\end{proof}
\begin{rema}
We recover the first Banach–Ne\v{c}as–Babu\v{s}ka condition \cite[Thm. 25.9, (BNB1)]{ErGu21-II}:
\[
a_S\,\left(\,(\uvec',p'),\,(\vvec^\star,q^\star)\,\right)\geq \frac{\nu}{2}\,C_{\min}\,(C_{\max})^{-1}\,\|(\uvec',p')\|_{\Xcal}\,\|(\vvec^\star,q^\star)\|_{\Xcal}.
\]
\end{rema}
Let us call $C_{\mathrm{stab}}=\frac{\nu}{2}\,C_{\min}\,(C_{\max})^{-1}$ the stability constant. With the choice of our parameters, $C_{\mathrm{stab}}$ is such that:
\[
C_{\mathrm{stab}}=\left\{
\begin{array}{rl}
\ds\frac{\nu}{4}\,\frac{C_{\div}}{1+C_{\div}}&\mbox{if } 0<C_{\div}\leq 1,\\
\\
\ds\frac{\nu}{4}\,\frac{(C_{\div})^{-1}}{1+C_{\div}}&\mbox{if } 1\leq C_{\div}.
\end{array}\right.
\]
Thus, the T-coercivity approach allows to give an estimate of the stability constant. In our computations, it depends on the choice of the parameters $\eta$ and $\gamma$, so that it could be optimized.\\  
If we were using a conforming discretization to solve Problem \eqref{eq:Stokes-FVa} (a.e. Taylor-Hood finite elements \cite{TaHo73}), we would use the bilinear form $a_S(\cdot,\cdot)$ to state the discrete variational formulation. Let us call the discrete spaces $\bX_{c,h}\subset\bH^1_0(\Om)$ and $Q_{c,h}\subset L^2_{zmv}(\Om)$. Then to prove the discrete T-coercivity, we would need to state the discrete counterpart to Proposition \ref{prop:diviso}. To do so, we can build a linear operator $\Pi_c:\bX\rightarrow\bX_h$, known as Fortin operator, such that (see a.e. \cite[\S 8.4.1]{BoBF13}): 
\begin{align}
\label{eq:WellPosed-c-1}
\exists C_c\,|\,
\forall\vvec\in\bH^1(\Om)\hspace*{5mm}\|\Grad\Pi_c\vvec\|_{\bbL^2(\Om)}&\leq C_c\|\Grad\vvec\|_{\bbL^2(\Om)},\\
\label{eq:WellPosed-c-2}
\forall\vvec\in\bH^1(\Om)\quad(\div\Pi_c\vvec,q_h)_{L^2(\Om)}&=(\div\vvec,q_h)_{L^2(\Om)},\quad\forall q_h\in Q_{c,h}.
\end{align}
Using a nonconforming discretization, we will not use the bilinear form $a_S(\cdot,\cdot)$ to exhibit the discrete variational formulation, but we will need a similar operator to \eqref{eq:WellPosed-c-1}-\eqref{eq:WellPosed-c-2} to prove the discrete T-coercivity, which is stated in Theorem \ref{thm:WellPosed-nc}.
\section{Discretization}\label{sec:disc}
We call $(O,(x_{d'})_{d'=1}^d)$ the Cartesian coordinates system, of orthonormal basis $(\evec_{d'})_{d'=1}^d$. 
Consider $(\Tcal_h)_h$ a simplicial triangulation sequence of $\Om$. For a triangulation $\Tcal_h$, we use the following index sets:
\begin{itemize}
\item $\Ical_K$ denotes the index set of the elements, such that $\Tcal_h:=\ds\bigcup_{\ell\in\Ical_K}K_\ell$  is the set of elements.
\item $\Ical_F$ denotes the index set of the facets\footnote{The term facet stands for face (resp. edge) when $d=3$ (resp. $d=2$).}, such that $\Fcal_h:=\ds\bigcup_{f\in\Ical_F}F_f$ is the set of facets.
\item[]Let $\Ical_F=\Ical_F^i\cup\Ical_F^b$, where $\forall f\in\Ical_F^i$, \textcolor{blue}{$F_f\in\Om$} and $\forall f\in\Ical_F^b$, $F_f\in\pa\Om$.
\item $\Ical_S$ denotes the index set of the vertices, such that $(S_j)_{j\in\Ical_S}$ is the set of vertices.
\item[]Let $\Ical_S=\Ical_S^i\cup\Ical_S^b$, where $\forall j\in\Ical_S^i$, \textcolor{blue}{$S_j\in\Om$} and $\forall j\in\Ical_S^b$, $S_j\in\pa\Om$.
\end{itemize} 
We also define the following index subsets:
\begin{itemize}
\item $\forall\ell\in\Ical_K$, $\Ical_{F,\ell}=\{f\in\Ical_F\,|\,F_f\in K_\ell\},\quad\Ical_{S,\ell}=\{j\in\Ical_S\,|\,S_j\in K_\ell\}$.
\item $\forall j\in\Ical_S$, $\Ical_{K,j}=\{\ell\in\Ical_K\,|\,S_j\in K_\ell\}$, $\quad N_j:=\mathrm{card}(\Ical_{K,j})$.
\end{itemize}
For all $\ell\in\Ical_K$, we call $h_\ell$ and $\rho_\ell$ the diameters of $K_\ell$ and its inscribed sphere respectively, and we let: $\sigma_\ell=\frac{h_\ell}{\rho_\ell}$. When the $(\Tcal_h)_h$ is a shape-regular triangulation sequence  (see a.e. \cite[def. 11.2]{ErGu21-I}), there exists a constant $\sigma>1$ such that for all $h$, for all $\ell\in\Ical_K$, $\sigma_\ell\leq\sigma$. For all $f\in\Ical_F$, $M_f$ denotes the barycentre of $F_f$, and by $\nvec_f$ its unit normal (outward oriented if $F_f\in\pa\Omega$). For all $j\in\Ical_S$, for all $\ell\in\Ical_{K,j}$, $\lambda_{j,\ell}$ denotes the barycentric coordinate of $S_j$ in $K_\ell$; $F_{j,\ell}$ denotes the face opposite to vertex $S_j$ in element $K_\ell$, and $\xvec_{j,\ell}$ denotes its barycentre. We call $\Scal_{j,\ell}$ the outward normal vector of $F_{j,\ell}$ and of norm $|\Scal_{j,\ell}|=|F_{j,\ell}|$. We remind the expression of $\lambda_{j,\ell}$ and the integration formula (25.14) p. 187 of \cite{Ciar91}:
\begin{equation}\label{eq:coorbary}
\begin{array}{rcl}
\ds\forall\xvec\in K_\ell,\,\lambda_{j,\ell}(\xvec)&=&\ds(d\,|K_\ell|)^{-1}\,(\xvec_{j,\ell}-\xvec)\cdot\Scal_{j,\ell}~;\\
\\
\ds\int_{K_\ell}\prod_{i\in\Ical_{S,\ell}}\lambda_{i,\ell}^{\alpha_{i\ell}}&=&
\ds|K_\ell|\,\frac{d!\,\ds\prod_{i\in\Ical_{S,\ell}}\,\alpha_{i,\ell}!}{\left(d+ \ds\sum_{i\in\Ical_{S,\ell}}\alpha_{i,\ell}\right)!}.
\end{array}
\end{equation}
Let introduce spaces of piecewise regular elements:\\
We set $\Pcal_h H^1=\left\{v\in L^2(\Om)\,;\quad\forall \ell\in\Ical_K,\,v_{|K_\ell}\in H^1(K_\ell)\right\}$, endowed with the scalar product~:
\[
(v,w)_h:=\sum_{\ell\in\Ical_K}(\grad v,\grad w)_{\bL^2(K_\ell)}\quad
\|v\|_h^2=\sum_{\ell\in\Ical_K}\|\grad v\|^2_{\bL^2(K_\ell)}.
\]
We set $\Pcal_h\bH^1=[\Pcal_hH^1]^d$, endowed with the scalar product~:
\[
(\vvec,\wvec)_h:=\sum_{\ell\in\Ical_K}(\Grad\vvec,\Grad\wvec)_{\bbL^2(K_\ell)}\quad
\|\vvec\|_h^2=\sum_{\ell\in\Ical_K}\|\Grad \vvec\|^2_{\bbL^2(K_\ell)}.
\]
Let $f\in\Ical_F^i$ such that $F_f=\pa K_L\cap\pa K_R$ and $\nvec_f$ is outward $K_L$ oriented.\\
The jump (resp. average) of a function $v\in \Pcal_h H^1$ across the facet $F_f$ is defined as follows: $[v]_{F_f}:=v_{|K_L}-v_{|K_R}$ (resp. $\{v\}_{F_f}:=\frac{1}{2}(v_{|K_L}+v_{|K_R})\,$). For $f\in\Ical_F^b$, we set: $[v]_{F_f}:=v_{|F_f}$ and $\{v\}_{F_f}:=v_{|F_f}$.\\
We set $\Pcal_h\bH(\div)=\left\{\vvec\in \bL^2(\Omega)\,;\quad\forall\ell\in\Ical_K,\,\vvec_{|K_\ell}\in\bH(\div;\,K_\ell)\right\}$, and we define the operator $\div_h$ such that: 
\[
\forall\vvec\in\Pcal_h\bH(\div),\,\forall q\in L^2(\Om),\quad
(\div_h\vvec,q)=\sum_{\ell\in\Ical_K}(\div\vvec,q)_{L^2(K_\ell)}.
\]
We recall classical finite elements estimates \cite{ErGu21-I}.
Let $\hatK$ be the reference simplex and $\hatF$ be the reference facet. For $\ell\in\Ical_K$ (resp. $f\in\Ical_F$), we denote by $T_\ell:\hatK\rightarrow K_\ell$ (resp. $T_f:\hatF\rightarrow F_f$) the geometric mapping such that $\forall\hat{\xvec}\in\hatK$, $\xvec_{|K_\ell}=T_\ell(\hat{\xvec})=\bbB_\ell\hat{\xvec}+\bvec_\ell$ (resp. $\xvec_{|F_f}=T_f(\hat{\xvec})=\bbB_f\hat{\xvec}+\bvec_f$), and we set $J_\ell=\mathrm{det}(\bbB_\ell)$ (resp. $J_f=\mathrm{det}(\bbB_f)$). There holds:
\[
|J_\ell|=d!\,|K_\ell|,\quad\|\bbB_\ell\|=\frac{h_\ell}{\rho_{\hatK}},\quad\|{\bbB_\ell}^{-1}\|=\frac{h_{\hatK}}{\rho_\ell},\quad |J_f|=(d-1)!\,|F_f|.
\]
For $v\in L^2(K_\ell)$, we set $\hat{v}_\ell=v\circ T_\ell$. For $v\in v^2(F_f)$, we set: $\hat{v}_f=v\circ T_f$. Changing the variable, we get: 
\begin{equation}\label{eq:vKvF}
\|v\|_{L^2(K_\ell)}^2=|J_\ell|\,\|\hat{v}_\ell\|_{L^2(\hatK)}^2,
\quad\mbox{and}\quad
\|v\|_{L^2(F_f)}^2=|J_f|\,\|\hat{v}_f\|_{L^2(\hatF)}^2.
\end{equation}
Let $v\in \Pcal_h H^1$. By changing the variable, $\grad v_{|K_\ell}=({\bbB_\ell}^{-1})^T\grad_{\hat{\xvec}}\hat{v}_\ell$, and it holds:
\begin{equation}\label{eq:gradvK}
\begin{array}{rcl}
\|\grad_{\hat{\xvec}}\hat{v}_\ell\|_{\bL^2(\hatK)}^2&\leq&\norm{\bbB_\ell}^2\,|J_\ell|^{-1}\,\|\grad v\|_{\bL^2(K_\ell)}^2,\\
&\lesssim&\sigma^2\,(\rho_\ell)^{-(d-2)}\,\|\grad v\|_{\bL^2(K_\ell)}.
\end{array}
\end{equation}
Let us recall some useful inequalities that we will need:
\begin{itemize}
\item The Poincaré-Steklov inequality in cells \cite[Lemma 12.11]{ErGu21-I}:\\
for all $\ell\in\Ical_K$ ($K_\ell$ is a convex set), $\forall v\in H^1(K_\ell)$:
\begin{equation}\label{eq:PSinC}
\|\ul{v}_\ell\|_{L^2(K_\ell)}\leq\pi^{-1}h_\ell\|\grad v\|_{\bL^2(K_\ell)},\quad\mbox{where }\ul{v}_\ell=v_{|K_\ell}-\frac{\int_{K_\ell}v}{|K_\ell|}.
\end{equation}
\item The multiplicative trace inequality as written in the proof of \cite[lemma 12.15]{ErGu21-I} for $p=2$: for all $\ell\in\Ical_K$, for all $f\in\Ical_{F,\ell}$, $\forall v\in H^1(K_\ell)$:
\begin{equation}\label{eq:MTI}
\|v\|_{L^2(F_f)}^2\leq\frac{|F_f|}{|K_\ell|}\|v\|_{L^2(K_\ell)}\left(\|v\|_{L^2(K_\ell)}+\frac{2}{d}l_{(\ell,f)}^{\perp}\|\grad v\|_{\bL^2(K_\ell)}\right),
\end{equation}
where $l_{(\ell,f)}^{\perp}$ is the largest length of an edge in $K_\ell$ and not belonging to $F_f$.
\item Combining \eqref{eq:PSinC} and \eqref{eq:MTI}, we get that $\forall v\in H^1(K_\ell)$:
\begin{equation}\label{eq:MTI-PSinC}
\|\ul{v}_\ell\|_{L^2(F_f)}^2\leq\frac{|F_f|}{|K_\ell|}\pi^{-1}h_\ell\,\left(\pi^{-1}h_\ell+\frac{2}{d}l_{(\ell,f)}^{\perp}\right)\,\|\grad v\|_{\bL^2(K_\ell)}^2.
\end{equation}
\end{itemize} 
Notice that in the reference element, inequality \eqref{eq:MTI-PSinC} reads:
\begin{equation}\label{eq:MTI-PSinC-ref}
\forall\ell\in\Ical_K,\,\forall f\in\Ical_{F,\ell},\quad\|\widehat{(\ul{v}_\ell)}_f\|_{L^2(\hatF)}^2\lesssim\|\grad_{\hat{\xvec}}\hat{v}_\ell\|_{\bL^2({\hatK})}^2.
\end{equation}
For all $D\subset\R^d$, we call $P^k(D)$ the set of order $k$ polynomials on $D$, $\bP^k(D)=(P^k(D))^d$, and we consider the broken polynomial space:
\[
P^k_{disc}(\Tcal_h)=\left\{q\in L^2(\Om);\quad \forall\ell\in\Ical_K,\,q_{|K_\ell}\in P^k(K_\ell)\right\}.
\]
\section{The nonconforming mixed finite element method for Stokes}\label{sec:Tcoerh}
The nonconforming finite element method was introduced by Crouzeix and Raviart in \cite{CrRa73} to solve Stokes Problem \eqref{eq:Stokes}. We approximate the vector space $\bH^1(\Om)$ component by component by piecewise polynomials of order $k\in\N^\star$. 
Let us consider $X_h$ (resp. $X_{0,h}$), the space of nonconforming approximation of $H^1(\Om)$ (resp. $H^1_0(\Om)$) of order $k$:
\begin{equation}
\label{eq:Xh}
\begin{array}{c}
X_h=\ds\left\{v_h\in P^k_{disc}(\Tcal_h)\,;\quad\forall f\in\Ical_F^i,\,\forall q_h\in P^{k-1}(F_f),\,\int_{F_f}[v_h]\,q_h=0\right\}\,;\\
\\
X_{0,h}=\ds\left\{v_h\in X_h\,;\quad\forall f\in\Ical_F^b,\,\forall q_h\in P^{k-1}(F_f),\,\ds\int_{F_f} v_h\,q_h=0\right\}.
\end{array}
\end{equation}
The condition on the jumps of $v_h$ on the inner facets is often called the patch-test condition.\\
\begin{prop}\label{pro:broknorm}
The broken norm $v_h\rightarrow\|v_h\|_h$ is a norm over $X_{0,h}$. 
\end{prop}
\begin{proof}
Let $v_h\in X_{0,h}$ such that $\|v_h\|_h=0$. Then for all $\ell\in\Ical_K$, $v_{h|K_\ell}$ is a constant. For all $f\in\Ical_F^i$ the jump $[v_h]_{F_f}$ vanishes, so that $v_h$ is a constant over $\Om$. We deduce from the discrete boundary condition that $v_h=0$. 
\end{proof}
The space of nonconforming approximation of $\bH^1(\Om)$ (resp. $\bH^1_0(\Om)$) of order $k$ is $\bX_h=(X_h)^d$ (resp. $\bX_{0,h}=(X_{0,h})^d$). We set $\Xcal_h:=\bX_{0,h}\times Q_h$ where $Q_h=P^{k-1}_{disc}(\Tcal_h)\cap L^2_{zmv}(\Om)$. We deduce from Proposition \ref{pro:broknorm} the
\begin{prop}
The broken norm defined below is a norm on $\Xcal_h$:
\begin{equation}
\|(\cdot,\cdot)\|_{\Xcal_h}:\left\{\begin{array}{rcl}\Xcal_h&\mapsto&\R \\(\vvec_h,q_h)&\rightarrow&\ds\|\vvec_h\|_h+\nu^{-1}\,\|q_h\|_{L^2(\Om)}
\end{array}\right..
\end{equation}
\end{prop}
Thus, the product space $\Xcal_h$ endowed with the broken norm $\|\cdot\|_{\Xcal_h}$ is a Hilbert space. 
\begin{prop}
The following discrete Poincaré–Steklov inequality holds: there exists a constant $C_{PS}^{nc}$ independent of $\Tcal_h$ such that
\begin{equation}\label{eq:cps-constant}
\forall\vvec_h\in\bX_{0,h},\quad\|\vvec_h\|_{\bL^2(\Om)}\leq C_{PS}^{nc}\,\|\vvec_h\|_h,
\end{equation}
where $C_{PS}^{nc}$ is independent of $\Tcal_h$ and is proportional to the diameter of $\Om$.
\end{prop}
\begin{proof}
\textcolor{blue}{
Inequality \eqref{eq:cps-constant} is stated in \cite[Lemma 36.6]{ErGu21-II} for $k=1$, but one can check that the proof holds true for higher orders, thanks to the patch-test condition. An alternative proof is given in \cite[Theorem C.1]{Saut22}.}
\end{proof}
We consider the discrete continuous bilinear form $a_{S,h}(\cdot,\cdot)$ such that~:
\[
\left\{
\begin{array}{rcl}
a_{S,h}:\Xcal_h\times\Xcal_h&\rightarrow&\R\\
(\uvec'_h,p'_h)\times(\vvec_h,q_h)&\mapsto&\nu(\uvec'_h,\vvec_h)_h-(p'_h,\div_h\vvec_h)-(q_h,\div_h\uvec'_h)
\end{array}\right..
\]
Let $\ell_\fvec\in\Lcal(\Xcal_h,\R)$ be such that :
\[
\forall(\vvec_h,q_h)\in\Xcal_h,\quad\ell_\fvec\,\left(\,(\vvec_h,q_h)\,\right)=\left\{
\begin{array}{rl}
(\fvec,\vvec_h)_{\bL^2(\Om)}&\mbox{if }\fvec\in\bL^2(\Om)\\
\langle\fvec,\Ical_h(\vvec_h)\rangle_{\bH^1_0(\Om)}&\mbox{if }\fvec\not\in\bL^2(\Om)
\end{array}\right.,
\]
where $\Ical_h:\bX_{0,h}\rightarrow\bY_{0,h}$, with $\bY_{0,h}=\{\vvec_h\in\bH^1_0(\Om)\,;\quad\forall\ell\in\Ical_K,\,\vvec_{h|K_\ell}\in \bP^k(K_\ell)\}$, is the averaging operator described in \cite[\S 22.4.1]{ErGu21-I}. There exists a constant $C^{nc}_{\Ical_h}>0$ independent of $\Tcal_h$ such that~:
\begin{equation}
\label{eq:c-constant}
\|\Ical_h\vvec_h\|_{\bH^1_0(\Om)}\leq C^{nc}_{\Ical_h}\,\|\vvec_h\|_h,\quad\forall\vvec_h\in\bX_{0,h}.
\end{equation}
The nonconforming discretization of Problem \eqref{eq:Stokes-FVa} reads:\\
Find $(\uvec_h,p_h)\in\Xcal_h$ such that
\begin{equation}\label{eq:Stokes-FVah}
\quad a_{S,h}\,\left((\uvec_h,p_h),(\vvec_h,q_h)\right)=\ell_\fvec\,\left(\,(\vvec_h,q_h)\,\right)\quad\forall(\vvec_h,q_h)\in\Xcal_h.
\end{equation}
Let us prove that Problem \eqref{eq:Stokes-FVah} is well-posed using the T-coercivity theory.
\begin{theo}
\label{thm:WellPosed-nc}
Suppose that there exists a Fortin operator $\Pi_{nc}:\bH^1(\Om)\rightarrow\bX_h$ such that 
\textcolor{blue}{
\begin{align}
\label{eq:WellPosed-nc-1}
\exists C_{nc}\,|\,
\forall\vvec\in\bH^1(\Om)\hspace*{5mm}\|\Pi_{nc}\vvec\|_h&\leq C_{nc}\|\Grad\vvec\|_{\bbL^2(\Om)},\\
\label{eq:WellPosed-nc-2}
\forall\vvec\in\bH^1(\Om)\quad(\div_h\Pi_{nc}\vvec,q_h)&=(\div\vvec,q_h)_{L^2(\Om)},\quad\forall q\in Q_h,
\end{align}
}
where the constant $C_{nc}$ does not depend on $h$. Then Problem \eqref{eq:Stokes-FVah} is well-posed. Moreover, it admits one and only one solution $(\uvec_h,p_h)$ such that:
\begin{equation}\label{eq:stabiNC}
\begin{array}{rl}
\mbox{\emph{if }}\fvec\in\bL^2(\Om):&
\left\{
\begin{array}{rcl}
\|\uvec_h\|_h&\leq& C_{PS}^{nc}\,\nu^{-1}\,\|\fvec\|_{\bL^2(\Om)}\\
\\
\|p_h\|_{L^2(\Om)}&\leq& 2\,C_{PS}^{nc}\,C_{\div}^{nc}\,\|\fvec\|_{\bL^2(\Om)}
\end{array}\right.,\\
\\
\mbox{\emph{if }}\fvec\not\in\bL^2(\Om):&\left\{
\begin{array}{rcl}
\|\uvec_h\|_h&\leq&C_{\Ical_h}^{nc}\,\nu^{-1}\,\|\fvec\|_{\bH^{-1}(\Om)}\\
\\
\|p_h\|_{L^2(\Om)}&\leq& 2\,C_{\Ical_h}^{nc}\,C_{\div}^{nc}\,\|\fvec\|_{\bH^{-1}(\Om)}
\end{array}\right.,
\end{array}
\end{equation}
where $C_{\div}^{nc}=C_{\div}\,C_{nc}$. 
Additionally, we can compute classical a priori error estimates (see \cite[Theorems 3, 4 and 6]{CrRa73}). Suppose that $(\uvec,p)\in\bH^{1+k}(\Om)\times H^k(\Om)$, we then have the estimate:
\begin{equation}
\label{eq:aprioriEE}
\|\uvec-\uvec_h\|_{\bL^2(\Om)}\leq C\sigma^\ell\,h^{k+1}\,
\left(|\uvec|_{\bH^{k+1}(\Om)}+\nu^{-1}\,|p|_{H^k(\Om)}\right),
\end{equation}
where the constant $C>0$ is independent of $h$, $\sigma$ is the shape regularity constant and the exponent $\ell\in\N^\star$ depends on $k$.
\end{theo}
\begin{proof}
Let us consider $(\uvec'_h,p'_h)\in\Xcal_h$ and let us build $(\vvec_h^\star,q_h^\star)\in\Xcal_h$ satisfying \eqref{eq:Tcoer-sym} (with $V=\Xcal_h$). We follow the three main steps of the proof of Theorem \ref{thm:WellPosed}. 
\begin{itemize}
\item[1.]
According to Proposition \ref{prop:diviso}, there exists $\vvec_{p'_h}\in\bV^{\perp}$ such that $\div\vvec_{p'_h}=\ds\nu^{-1}\,p'_h$ in $\Om$ and:
\[
\|\vvec_{p'_h}\|_{\bH^1_0(\Om)}^2\leq \left(\frac{C_{\div}}{\nu}\right)^2\,\|p'_h\|_{L^2(\Om)}^2.
\]
Consider $\vvec_{h,p'_h}=\Pi_{nc}\vvec_{p'_h}$, for all $q_h\in Q_h$, we have:  $(\div_h\vvec_{h,p'_h},q_h)=\ds\nu^{-1}\,(p'_h,q_h)_{L^2(\Om)}$ and
\begin{equation}
\label{eq:GiRa-nc}
\|\vvec_{h,p'_h}\|_h^2\leq \left(\frac{C_{\div}^{nc}}{\nu}\right)^2\,\|p'_h\|_{L^2(\Om)}^2\mbox{ where }C_{\div}^{nc}=C_{nc}\,C_{\div}.
\end{equation}
Let us set $(\vvec_h^\star,q_h^\star):=(\gamma_{nc}\uvec'_h-\vvec_{h,p'_h},-\gamma_{nc}\,p'_h)$, with $\gamma_{nc}>0$. We obtain:
\begin{equation}\label{eq:coer1-nc}
a_{S,h}\,\left(\,(\uvec'_h,p'_h),\,(\vvec_h^\star,q_h^\star)\,\right)=
\nu\,\gamma_{nc}\|\uvec'_h\|_h^2+\nu^{-1}\| p'_h\|_{L^2(\Om)}^2-\nu(\uvec'_h,\vvec_{h,p'_h})_h.
\end{equation}
\item[2.] In order to bound the last term of \eqref{eq:coer1-nc}, we use Young inequality and then inequality \eqref{eq:GiRa-nc}:
\begin{equation}\label{eq:GvGvp-nc}
(\uvec'_h,\vvec_{h,p'_h})_h\leq
\ds\frac{\eta_{nc}}{2}\|\uvec'_h\|_h^2+\frac{\eta_{nc}^{-1}}{2}\,\left(\frac{C_{\div}^{nc}}{\nu}\right)^2\,\| p'_h\|_{L^2(\Om)}^2.
\end{equation}
\item[3.] 
Using the bound \eqref{eq:GvGvp-nc} in \eqref{eq:coer1-nc} and choosing $\ds\eta_{nc}=\gamma_{nc}$, we get:
\[
a_{S,h}\,\left(\,(\uvec'_h,p'_h),\,(\vvec_h^\star,q_h^\star)\,\right)\geq\frac{\gamma_{nc}}{2}\nu\,\|\uvec'_h\|_h^2+\nu^{-1}\left(1+\frac{(\gamma_{nc})^{-1}}{2}\,(C_{\div}^{nc})^2\,\right)\| p'_h\|_{L^2(\Om)}^2
\]
Consider now $\gamma_{nc}=(C_{\div}^{nc})^2$. \\
Noticing that $\nu\|\uvec_h\|_h^2+\nu^{-1}\|p'_h\|_{L^2(\Om)}^2\geq\ds\frac{\nu}{2}\|(\uvec'_h,p'_h)\|_{\Xcal_h}^2$, we obtain:
\[
a_{S,h}\,\left(\,(\uvec'_h,p'_h),\,(\vvec_h^\star,q_h^\star)\,\right)\geq \frac{\nu}{2}C_{\min}^{nc}\,\|(\uvec'_h,p'_h)\|_{\Xcal_h}^2,
\]
\end{itemize}
where $C_{\min}^{nc}=\ds\frac{1}{2}\,\min(\,(C_{\div})^2,1\,)$.\\
The operator $T_h$ such that $T_h(\uvec_h',p'_h)=(\vvec^\star_h,p^\star_h)$ is linear and continuous: \[
\|T_h(\uvec_h',p'_h)\|_{\Xcal_h}=\|\vvec_h^\star\|_h+\nu^{-1}\,\|q_h^\star\|_{L^2(\Omega)}\leq C_{\max}^{nc}\,\|(\uvec_h',p_h')\|_{\Xcal_h}^2
\]
where $C_{\max}^{nc}=C_{\div}^{nc}\,(C_{\div}^{nc}+1)$. \footnote{\textcolor{blue}{Note that $(\vvec_h^\star,q_h^\star)=({\bf{0}},0)\Leftrightarrow(\uvec'_h,p'_h)=({\bf{0}},0)$, so that the operator $T_h\in\Lcal(\Xcal_h,\Xcal_h)$ is bijective.}}
The discrete continuous bilinear form $a_{S,h}(\cdot,\cdot)$ is then $T_h$-coercive and according to Theorem \ref{thm:WellPosed}, Problem \eqref{eq:Stokes-FVah} is well posed.  Consider $(\uvec_h,p_h)$ the unique solution of Problem \eqref{eq:Stokes-FVah}. Choosing $\vvec_h=0$, we obtain that $\div_h\uvec_h=0$. Now, choosing $\vvec_h=\uvec_h$ in \eqref{eq:Stokes-FVah} and using Cauchy-Schwarz inequality, we get that:
\begin{equation}\label{eq:uh_bound}
\left\{\begin{array}{rcll}
\|\uvec_h\|_h&\leq& \nu^{-1}\,C_{PS}^{nc}\,\|\fvec\|_{\bL^2(\Om)}&\mbox{if }\fvec\in\bL^2(\Om),\mbox{ using \eqref{eq:cps-constant}}~;\\
\\
\|\uvec_h\|_h&\leq& \nu^{-1}\,C_{\Ical_h}^{nc}\,\|\fvec\|_{\bH^{-1}(\Om)}&\mbox{if }\fvec\not\in\bL^2(\Om),\mbox{ using \eqref{eq:c-constant}}.
\end{array}\right.
\end{equation}
Consider $(\vvec_h,q_h)=(\vvec_{h,p_h},0)$ in \eqref{eq:Stokes-FVah}, where $\vvec_{h,p_h}=\Pi_{nc}\vvec_{p_h}$ is built as $\vvec_{h,p'_h}$ in point $1$, setting $p'_h=p_h$. Suppose that $\fvec\in\bL^2(\Om)$. Using the triangular inequality, Cauchy-Schwarz inequality, Poincaré-Steklov inequality \eqref{eq:cps-constant}, Theorem \ref{thm:WellPosed-nc}, and estimate \eqref{eq:uh_bound}, we have:  
\[
\begin{array}{rcl}
\|p_h\|_{L^2(\Om)}^2&=&\nu\,(\uvec_h,\vvec_{h,p_h})_h-(\fvec,\vvec_{h,p_h})_{\bL^2(\Om)}\,,\\
\\
&\leq&\nu\,\|\uvec_h\|_h\,\|\vvec_{h,p_h}\|_h+\|\fvec\|_{\bL^2(\Om)}\,\|\vvec_{h,p_h}\|_{\bL^2(\Om)}\\
\\
&\leq&2\,C_{PS}^{nc}\,\|\fvec\|_{\bL^2(\Om)}\,\|\vvec_{h,p_h}\|_h\leq 2\,C_{PS}^{nc}\,C_{nc}\,\|\fvec\|_{\bL^2(\Om)}\,\|\Grad\vvec_{p_h}\|_{\bbL^2(\Om)}\,,\\
\\
&\leq&2\,C_{PS}^{nc}\,C_{\div}^{nc}\,\|\fvec\|_{\bL^2(\Om)}\,\|p_h\|_{L^2(\Om)}.
\end{array}
\]
Applying the same reasoning when $\fvec\in\bH^{-1}(\Om)$, we get that:
\begin{equation}\label{eq:ph_bound}
\left\{\begin{array}{rcll}
\|p_h\|_{L^2(\Om)}&\leq&2\,C_{PS}^{nc}\,C^{nc}_{\div}\,\|\fvec\|_{\bL^2(\Om)}&\mbox{if }\fvec\in\bL^2(\Om),\mbox{ using \eqref{eq:cps-constant}}~;\\
\\
\|p_h\|_{L^2(\Om)}&\leq&2\,C_{\Ical_h}^{nc}\,\,C^{nc}_{\div}\|\fvec\|_{\bH^{-1}(\Om)}&\mbox{if }\fvec\not\in\bL^2(\Om),\mbox{ using \eqref{eq:c-constant}}.
\end{array}\right.
\end{equation}
The a priori error estimate corresponds to \cite[Theorem 4]{CrRa73}.
\end{proof}
\begin{rema}
Again, we recover the first Banach–Ne\v{c}as–Babu\v{s}ka condition \cite[Thm. 25.9, (BNB1)]{ErGu21-II}:
\[
a_{S,h}\,\left(\,(\uvec_h',p_h'),\,(\vvec_h^\star,q_h^\star)\,\right)\geq \frac{\nu}{2}\,(C_{\max}^{nc})^{-1}\,\|(\uvec'_h,p'_h)\|_{\Xcal_h}\,\|(\vvec_h^\star,q_h^\star)\|_{\Xcal_h}.
\]
\end{rema}
Let us call $C_{\mathrm{stab}}^{nc}=\frac{\nu}{2}\,C_{\min}^{nc}\,(C_{\max}^{nc})^{-1}$ the stability constant. With the choice of our parameters, $C_{\mathrm{stab}}^{nc}$ is such that:
\[
C_{\mathrm{stab}}^{nc}=\left\{
\begin{array}{rl}
\ds\frac{\nu}{4}\,\frac{C_{\div}^{nc}}{1+C_{\div}^{nc}}&\mbox{if } 0<C_{\div}^{nc}\leq 1,\\
\\
\ds\frac{\nu}{4}\,\frac{(C_{\div}^{nc})^{-1}}{1+C_{\div}^{nc}}&\mbox{if } 1\leq C_{\div}^{nc}.
\end{array}\right.
\]
The main issue with nonconforming mixed finite elements is the construction the basis functions. In a recent paper, Sauter explains such a construction in two dimensions \cite[Corollary 2.4]{Saut22}, and gives a bound to the discrete counterpart $\beta_\Tcal(\Om)$ of $\beta(\Om)$ defined in \eqref{eq:LBB-CIS}:
\begin{equation}\label{eq:LBB-CISh}
\beta_\Tcal(\Om)=\inf_{q_h\in Q_h\backslash \{0\}}\sup_{\vvec_h\in\bX_{0,h}\backslash \{0\}}\frac{(\div_h\vvec_h)}{\|q_h\|_{L^2(\Om)}\,\|\vvec_h\|_h}\geq c_\Tcal\,k^{-\alpha}.
\end{equation}
This bound is in $c_\Tcal\,k^{-\alpha}$, where the parameter $\alpha$ is explicit and depends on $k$ and on the mesh topology; and the constant $c_\Tcal$ depends only on the shape-regularity of the mesh. 
\section{Nonconforming Crouzeix-Raviart mixed finite elements}\label{sec:CrRa}
We study the lowest order nonconforming Crouzeix-Raviart mixed finite elements \cite{CrRa73}. Let us consider $X_{CR}$ (resp. $X_{0,CR}$), the space of nonconforming approximation of $H^1(\Om)$ (resp. $ H^1_0(\Om)$) of order $1$:
\begin{equation}
\label{eq:XCR}
\begin{array}{rcl}
\ds X_{CR}&=&\ds\left\{v_h\in P^1_{disc}(\Tcal_h)\,;\quad
\forall f\in\Ical_F^i,\,\int_{F_f}[v_h]=0\right\}\,;\\
\\
\ds X_{0,CR}&=&\ds\left\{v_h\in X_{CR}\,;\quad\forall f\in\Ical_F^b,\,\ds\int_{F_f} v_h=0\right\}.
\end{array}
\end{equation}
The space of nonconforming approximation of of $\bH^1(\Om)$ (resp. $\bH^1_0(\Om)$) of order $1$ is $\bX_{CR}=(X_{CR})^d$ (resp. $\bX_{0,CR}=(X_{0,CR})^d$). We set $\Xcal_{CR}:=\bX_{0,CR}\times Q_{CR}$ where $Q_{CR}=P^0_{disc}(\Tcal_h)\cap L^2_{zmv}(\Om)$. \\
We can endow $X_{CR}$ with the basis $(\psi_f)_{f\in\Ical_F}$ such that: $\forall \ell\in\Ical_K$,
\[
\psi_{f|K_\ell}=\left\{\begin{array}{cl}1-d\lambda_{i,\ell}&\mbox{if }f\in\Ical_{F,\ell},\\ 0&\mbox{ otherwise,}\end{array}\right.
\]
where $S_i$ is the vertex opposite to $F_f$ in $K_\ell$.  We then have $\psi_{f|F_f}=1$, so that $[\psi_f]_{F_f}=0$ if $f\in\Ical_F^i$ (i.e. $F_f\in\overset{\circ}{\Om}$), and $\forall f'\neq f$, $\int_{F_{f'}}\psi_f=0$. \\
We have: $X_{CR}=\mathrm{vect}\left((\psi_f)_{f\in\Ical_F}\right)$ and $X_{0,CR}=\mathrm{vect}\left((\psi_f)_{f\in\Ical_F^i}\right)$.\\ 
The Crouzeix-Raviart interpolation operator $\pi_{CR}$ for scalar functions is defined by:
\[
\pi_{CR}: \left\{\begin{array}{rcl} H^1(\Om)&\rightarrow&X_{CR}\\ v&\mapsto&\ds\sum_{f\in\Ical_F}\pi_fv\,\psi_{f}\end{array}\right.,\mbox { where }\pi_fv=\frac{1}{|F_f|}\int_{F_f} v.
\]
Notice that $\forall f\in\Ical_F$, $\int_{F_f} \pi_{CR}v=\int_{F_f} v$. Moreover, the Crouzeix-Raviart interpolation operator preserves the constants, so that $\pi_{CR}\ul{v}_\Om=\ul{v}_\Om$ where $\ul{v}_\Om=\int_\Om v/|\Om|$. We recall the following result \cite[Lemma 2]{ApNS01}):
\begin{lemm}\label{lem:piCR}
The Crouzeix-Raviart interpolation operator $\pi_{CR}$ is such that:
\begin{equation}\label{eq:gradvhleqgradv-glo}
\forall v\in H^1(\Om),\quad
\|\pi_{CR}v\|_h\leq \|\grad v\|_{\bL^2(\Om)}.
\end{equation}
\end{lemm}
\begin{proof}
We have, integrating by parts twice and using Cauchy-Schwarz inequality:
\[
\begin{array}{rcl}
\grad\pi_{CR}v{}_{|K_\ell}&=&
\ds|K_\ell|^{-1}\int_{K_\ell}\grad\pi_{CR}v=|K_\ell|^{-1}\sum_{f\in\Ical_{F,\ell}}\int_{F_f}\pi_{CR}v\,\nvec_f,\\
&=&\ds|K_\ell|^{-1}\sum_{f\in\Ical_{F,\ell}}\int_{F_f}v\,\nvec_f=|K_\ell|^{-1}\int_{K_\ell}\grad v,\\
|\grad\pi_{CR}v{}_{|K_\ell}|&\leq&\ds|K_\ell|^{-1/2}\,\|\grad v\|_{\bL^2(K_\ell)}\\
\Rightarrow\quad\|\grad \pi_{CR}v\|_{\bL^2(K_\ell)}^2&\leq&\ds\|\grad v\|_{\bL^2(K_\ell)}^2.
\end{array}
\]
Summing these local estimates over $\ell\in\Ical_K$, we obtain \eqref{eq:gradvhleqgradv-glo}. 
\end{proof}
For a vector $\vvec\in\bH^1(\Om)$ of components $(v_{d'})_{d'=1}^d$, the Crouzeix-Raviart interpolation operator is such that: $\Pi_{CR}\vvec=\left(\pi_{CR}v_{d'}\right)_{d'=1}^d$. Let us set $\Pi_f\vvec=\left(\pi_fv_{d'}\right)_{d'=1}^d$.
\begin{lemm}\label{lem:PiCR}
The Crouzeix-Raviart interpolation operator $\Pi_{CR}$ can play the role of the Fortin operator:
\begin{align}
\label{eq:WellPosed-CR-1}
\forall\vvec\in\bH^1(\Om)\hspace*{5mm}\|\Pi_{CR}\vvec\|_h&\leq\|\Grad\vvec\|_{\bbL^2(\Om)},\\
\label{eq:WellPosed-CR-2}
\forall\vvec\in\bH^1(\Om)\quad(\div_h\Pi_{CR}\vvec,q_h)&=(\div\vvec,q_h)_{L^2(\Om)},\quad\forall q\in Q_h,
\end{align}
Moreover, for all $\vvec\in\bP^1(\Om)$, $\Pi_{CR}\vvec=\vvec$.
\end{lemm}
\begin{proof}
We obtain \eqref{eq:WellPosed-CR-1} applying Lemma \eqref{lem:piCR} component by component. By integrating by parts, we have $\forall\vvec\in\bH^1(\Om)$, $\forall\ell\in\Ical_K$:
\[
\begin{array}{rcl}
\ds\int_{K_\ell}\div\Pi_{CR}\vvec
&=&\ds\sum_{f\in\Ical_{F,\ell}}\int_{F_f}\Pi_{CR}\vvec\cdot\nvec_f
=\sum_{f\in\Ical_{F,\ell}}\int_{F_f}\Pi_f\vvec\cdot\nvec_f,\\
&=&\ds\sum_{f\in\Ical_{F,\ell}}\int_{F_f} \vvec\cdot\nvec_f
=\int_{K_\ell}\div\vvec,
\end{array}
\]
so that \eqref{eq:WellPosed-CR-2} is satisfied.
\end{proof}
We can apply the T-coercivity theory to show the next following result:
\begin{theo}
Let $\Xcal_h=\Xcal_{CR}$. Then the continuous bilinear form $a_{S,h}(\cdot,\cdot)$ is $T_h$-coercive and Problem \eqref{eq:Stokes-FVah} is well-posed.
\end{theo}
\begin{proof}Using estimates \eqref{eq:WellPosed-CR-1} and  \eqref{eq:cps-constant}, we apply the proof of Theorem \ref{thm:WellPosed-nc}.
\end{proof}
Since the constant of the interpolation operator $\Pi_{CR}$ is equal to $1$, we have $C_{min}^{CR}=C_{min}$ and $C_{max}^{CR}=C_{max}$: the stability constant of the nonconforming Crouzeix-Raviart mixed finite elements is independent of the mesh. This is not the case for higher order (see \cite[Theorem 2.2]{CaSa22}).
\section{Fortin-Soulie mixed finite elements}\label{sec:FoSo}
We consider here the case $d=2$ and we study Fortin-Soulie mixed finite elements \cite{FoSo83}. We consider a shape-regular triangulation sequence $(\Tcal_h)_h$.\\
Let us consider $X_{FS}$ (resp. $X_{0,FS}$), the space of nonconforming approximation of $H^1(\Om)$ (resp. $H^1_0(\Om)$) of order $2$:
\begin{equation}
\label{eq:XFS}
\begin{array}{c}
X_{FS}=\ds\left\{v_h\in P^2_{disc}(\Tcal_h)\,;\quad
\forall f\in\Ical_F^i,\,\forall q_h\in P^1(F_f),\,\int_{F_f}[v_h]\,q_h=0\right\}\,;\\
\\
X_{0,FS}=\ds\left\{v_h\in X_{FS}\,;\quad\forall f\in\Ical_F^b,\,\forall q_h\in P^1(F_f),\,\int_{F_f}v_h\,q_h=0\right\}.
\end{array}
\end{equation}
The space of nonconforming approximation of $\bH^1(\Om)$ (resp. $\bH^1_0(\Om)$) of order $2$ is $\bX_{FS}=(X_{FS})^2$ (resp. $\bX_{0,{FS}}=(X_{0,{FS}})^2$).
We set $\Xcal_{FS}=\bX_{0,FS}\times Q_{FS}$ where $Q_{FS}:=P^1_{disc}(\Tcal_h)\cap L^2_{zmv}(\Om)$.
\\
The building of a basis for $X_{0,FS}$ is more involved  than for $X_{0,CR}$ since we cannot use two points per facet as degrees of freedom. Indeed, for all $\ell\in K_\ell$, there exists a polynomial of order $2$ vanishing on the Gauss-Legendre points of the facets of the boundary $\pa K_\ell$. Let $f\in\Ical_F$. The barycentric coordinates of the two Gauss-Legendre points $(p_{+,f},p_{-,f})$ on $F_f$ are such that: \[
p_{+,f}=(c_+,c_-),\, p_{-,f}=(c_-,c_+),\mbox{ where }c_{\pm}=(1\pm 1/\sqrt{3})/2.
\]
These points can be used to integrate exactly order three polynomials:
\[
\forall g\in P^3(F_f),\,\int_{F_f} g=\frac{|F_f|}{2}\left(g(p_{+,f})+g(p_{-,f})\right).
\]
For all $\ell\in\Ical_K$, we define the quadratic function $\phi_{K_\ell}$ that vanishes on the six Gauss-Legendre points of the facets  of $K_\ell$ (see Fig. \ref{fig:GLpoints}):
\begin{equation}\label{eq:phiKl}
\phi_{K_\ell}:=2-3\sum_{i\in\Ical_{S,\ell}}\lambda_{i,\ell}^2\mbox{ such that}\quad \forall f\in\Ical_{F,\ell},\,\forall q\in P^1(F),\quad\int_{F_f}\phi_{K_\ell}\,q=0.
\end{equation}
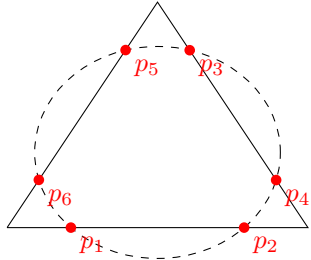
\begin{figure}[h!]
\centering
\begin{tikzpicture}
\draw[name path=line 1] (0,0)--(4,0)--(2,3)--(0,0);
\draw[dashed,name path=line 2] (2,1) ellipse (1.6330 and 1.41);
\fill[red,name intersections={of=line 1 and line 2,total=\t}]
   \foreach \s in {1,...,\t}{(intersection-\s) circle (2pt) node [below right] {$p_{\s}$}};
\end{tikzpicture}
    \caption{The six Gauss-Legendre points of an element $K_\ell$ and the elliptic function $\phi_{K_\ell}$.}
    \label{fig:GLpoints}
\end{figure}
\\
We also define the spaces of $P^2$-Lagrange functions:
\[
\begin{array}{rcl}
X_{LG}&:=&\ds\left\{\textcolor{blue}{v_h\in H^1(\Om)};\quad\forall \ell\in\Ical_K,\,v_{h|K_\ell}\in P^2(K_\ell)\right\},\\
 X_{0,LG}&:=&\ds\left\{v_h\in X_{LG};\quad v_{h|\pa\Om}=0\right\}.
\end{array}
\]
The Proposition below proved in \cite[Prop. 1]{FoSo83} allows to build a basis for $X_{0,FS}$:
\begin{prop}
We have the following decomposition: $X_{FS}=X_{LG}+\Phi_h$ with $\mathrm{dim}(X_{LG}\cap \Phi_h)=1$. Any function of $X_{FS}$ can be written as the sum of a function of $X_{LG}$ and a function of $\Phi_h$. This representation can be made unique by specifying one degree of freedom.
\end{prop}
\textcolor{blue}{
Notice that $\Phi_h\cap X_{LG}=\mathrm{vect}{(v_\Phi)}$, where for all $\ell\in\Ical_K$, $v_{\Phi|K_\ell}=\phi_{K_\ell}$. Then, counting the degrees of freedom, one can show that $\mathrm{dim}(X_{FS})=\mathrm{dim}(X_{LG})+\mathrm{dim}(\Phi_h)+1$. For problems involving Dirichlet boundary conditions we can prove thus that for $X_{0,FS}$ the representation is unique and $X_{0,FS}=X_{0,LG}\oplus\Phi_h$.} We have $X_{LG}=\mathrm{vect}\left((\phi_{S_i})_{i\in\Ical_S},(\phi_{F_f})_{f\in\Ical_F}\right)$ where the basis functions are such that: $\forall \ell\in\Ical_K$,
\begin{equation}\label{eq:phiSphiF}
\begin{array}{rcl}
\forall i\in\Ical_S,\quad\phi_{S_i|K_\ell}&=&\left\{\begin{array}{rl}
\lambda_{i,\ell}\,(2\lambda_{j,\ell}-1)&\quad\mbox{ if }i\in\Ical_{S,\ell}\\
0&\quad\mbox{ if }j\not\in\Ical_{S,\ell}
\end{array}\right.\\
\\
\forall f\in\Ical_F,\quad\phi_{F_f|K_\ell}&=&\left\{\begin{array}{rl}
4\,\lambda_{i,\ell}\,\lambda_{j,\ell}&\quad\mbox{ if }f\in\Ical_{F,\ell},\mbox{ and }F_f=S_iS_j\\
0&\quad\mbox{ if }f\not\in\Ical_{F,\ell}
\end{array}\right.
\end{array}.
\end{equation}
For all $\ell\in\Ical_K$, we will denote by $(\phi_{\ell,j})_{j=1}^6$ the local nodal basis such that: \[(\phi_{\ell,j})_{j=1}^3=(\phi_{S_i|K_\ell})_{i\in\Ical_{S,\ell}}\quad\mbox{and}\quad(\phi_{\ell,j})_{j=4}^6=(\phi_{F_f|K_\ell})_{f\in\Ical_{F,\ell}}.
\]
The spaces $X_{FS}$ and $X_{0,FS}$ are such that:
\begin{equation}\label{eq:XFSbasis}
\begin{array}{rcl}
X_{FS}&=&\ds\mathrm{vect}\left(\,(\phi_{S_i})_{i\in\Ical_S},(\phi_{F_f})_{f\in\Ical_F},(\phi_{K_\ell})_{\ell\in\Ical_K}\,\right)\,,\\
\\
X_{0,FS}&=&\ds\mathrm{vect}\left(\,(\phi_{S_i})_{i\in\Ical_S^i},(\phi_{F_f})_{f\in\Ical_F^i},(\phi_{K_\ell})_{\ell\in\Ical_K}\,\right).
\end{array}
\end{equation}
We propose here an alternative definition of the Fortin interpolation operator proposed in \cite{FoSo83}. Let us first recall the Scott-Zhang interpolation operator \cite{ScZh90,Ciar13}. For all $i\in\Ical_S$, we choose some $\ell_i\in\Ical_{K,i}$, and we build the $L^2(K_{\ell_i})$-dual basis $(\tilde{\phi}_{\ell_i,j})_{j=1}^6$ of the local nodal basis such that:
\[
\forall j,\,j'\in\{1,\cdots,6\},\quad\int_{K_{\ell_i}}\tilde{\phi}_{\ell_i,j}\,\phi_{\ell_i,j'}=\delta_{j,j'}. 
\]
Let us define the Fortin-Soulie interpolation operator for scalar functions by:
\begin{equation}
 \label{eq:piFS}
 \begin{array}{c}
\ds\pi_{FS}: \left\{\begin{array}{rcl} \Pcal_hH^1&\rightarrow&X_{FS}\\ v&\mapsto&\ds\tilde{\pi}v+\sum_{\ell\in\Ical_K}v_{K_\ell}\phi_{K_\ell}
\end{array}\right.,
\\
\ds\mbox{with}\quad\tilde{\pi}v=\sum_{i\in\Ical_S}v_{S_i}\phi_{S_i}+\sum_{f\in\Ical_F}\tilde{v}_{F_f}\phi_{F_f}.
 \end{array}
\end{equation}
\begin{itemize}
\item The coefficients $\left(v_{S_i}\right)_{i\in\Ical_S}$ are fixed so that: $\forall i\in\Ical_S$, $\ds v_{S_i}=\int_{K_{\ell,i}}v\,\tilde{\phi}_{\ell_i,j_i}$, where $j_i$ is the index such that $\ds\int_{K_{\ell_i}}\tilde{\phi}_{\ell_i,j_i}\,\phi_{S_i|K_{\ell_i}}=1$. Using Cauchy-Schwarz inequality and inequality \eqref{eq:vKvF}, we have:
\begin{equation}
\label{eq:estim1}
|v_{S_i}|\leq\left(\int_{K_{\ell_i}}\tilde{\phi}_{\ell_i,j_i}^2\right)^{1/2}\,\|v\|_{L^2(K_{\ell_i})}\lesssim|K_{\ell_i}|^{-1/2}\,\|v\|_{L^2(K_{\ell_i})}\lesssim\|\hat{v}_{\ell_i}\|_{L^2(\hat{K})}.
\end{equation}
\item The coefficients $\left(\tilde{v}_{F_f}\right)_{f\in\Ical_F}$ are fixed so that: $\forall f\in\Ical_F$, $\ds\int_{F_f}\tilde{\pi}\tilde{v}=\int_{F_f}\{v\}$. 
\end{itemize}
We then have:
\begin{equation}\label{eq:vF}
\tilde{v}_{F_f}=\frac{3}{2}\,v_{F_f}-\frac{1}{4}\sum_{i\in\Ical_{S,f}}v_{S_i},\mbox{ where }v_{F_f}:=\frac{1}{|F_f|}\int_{F_f}\{v\}
\end{equation}
For all $\ell\in\Ical_K$, the coefficient $v_{K_\ell}$ is fixed so that: $\ds\int_{K_\ell}\pi_{FS}v=\int_{K_\ell}v$.\\ 
The definition \eqref{eq:piFS} is more general than the one given in \cite{FoSo83}, which holds for $v\in H^2(\Om)$. \\
We set $\vvec_{S_i}:=\left(\,\tilde{\pi}v_1(S_i),\tilde{\pi}v_2(S_i)\right)^T$ and  $\tilde{\vvec}_{F_f}:=\left(\,\tilde{\pi}v_1(F_f),\tilde{\pi}v_2(F_f)\,\right)^T$.\\
We now define the Fortin-Soulie interpolation operator for vector functions by:
\[
\Pi_{FS}: \left\{\begin{array}{rcl} \bH^1(\Om)&\rightarrow&\bX_{FS}\\ 
\vvec&\mapsto&\ds\sum_{i\in\Ical_S}\vvec_{S_i}\phi_{S_i}+\sum_{f\in\Ical_F}\tilde{\vvec}_{F_f}\phi_{F_f}+\sum_{\ell\in\Ical_K}\vvec_{K_\ell}\phi_{K_\ell}.
\end{array}\right.
\]
For all $\ell\in\Ical_K$, the vector coefficient $\vvec_{K_\ell}\in\R^2$ is now fixed so that \textcolor{blue}{condition \eqref{eq:WellPosed-nc-2} is satisfied}. We can impose for example that the projection $\Pi_{FS}\vvec$ satisfies:
\begin{equation}
\label{eq:linsyst_vKl}
\int_{K_\ell} T_\ell^{-1}(\xvec)\,\div\Pi_{FS}\vvec=\int_{K_\ell} T_\ell^{-1}(\xvec)\,\div\vvec.
\end{equation}
\textcolor{blue}{Notice that due to \eqref{eq:phiKl}, the patch-test condition is still satisfied.} Moreover, one can check that for all $\vvec\in\bP^2(\Omega)$, $\Pi_{FS}\vvec=\vvec$.  In particular, if $\vvec\in\bP^1(\Om)$, we obtain that for all $\ell\in\Ical_K$, $\vvec_{K_\ell}=0$. Using definitions \eqref{eq:phiSphiF} and \eqref{eq:vF}, we obtain for all $\ell\in\Ical_K$:
\begin{equation}\label{eq:PiFS(v)}
\quad(\Pi_{FS}\vvec)_{|K_\ell}=\sum_{i\in\Ical_{S,\ell}}\vvec_{S_i}\,\psi_{S_i}+\frac{3}{2}\sum_{f\in\Ical_{F,\ell}}\vvec_{F_f}\phi_{F_f}+\vvec_{K_\ell}\phi_{K_\ell},
\end{equation}
where $\psi_{S_i|K_\ell}=3\,\lambda_{i,\ell}^2-2\,\lambda_{i,\ell}$.\\
Let us estimate $\vvec_{K_\ell}$. By changing the variable, setting $\hat{\vvec}_\ell(\hat{\xvec})=\vvec\circ T_\ell(\hat{\xvec})$, the linear system \eqref{eq:linsyst_vKl} is written as follows, for $d'\in\{1,2\}$:
\[
\begin{array}{rcl}
\ds({\bbB_\ell}^{-1}\,\vvec_{K_\ell})\cdot\int_{\hatK} \hat{x}_{d'}\,\grad_{\hat{\xvec}}\hat{\phi}_{K_\ell}&=&\ds\int_{\hatK}\hat{x}_{d'}\,\div_{\hat{\xvec}}({\bbB_\ell}^{-1}\hat{\vvec}_\ell)\\
&&-\ds\sum_{i\in \Ical_{S,\ell}}({\bbB_\ell}^{-1}\,\vvec_{S_i})\cdot\int_{\hatK}\hat{x}_{d'}\,\grad_{\hat{\xvec}}\hat{\psi}_{S_i}\\
&&\ds-\frac{3}{2}\sum_{f\in\Ical_{F,\ell}}({\bbB_\ell}^{-1}\,\vvec_{F_f})\cdot\int_{\hatK}\hat{x}_{d'}\,\grad_{\hat{\xvec}}\hat{\phi}_{F_f}.
\end{array}
\]
Noticing that $\int_{\hatK} \hat{x}_{d'}\,\grad_{\hat{\xvec}}\hat{\phi}_{K_\ell}=-\frac{1}{4}\evec_{d'}$, we have:
\begin{equation}\label{eq:vKl}
\begin{array}{rcl}
\ds\frac{1}{4}{\bbB_\ell}^{-1}\,\vvec_{K_\ell}&=&\ds\sum_{i\in \Ical_{S,\ell}}\int_{\hatK}\hat{\xvec}\,({\bbB_\ell}^{-1}\,\vvec_{S_i})\cdot\grad_{\hat{\xvec}}\hat{\psi}_{S_i}\\
\\
&&\ds+\frac {3}{2}\sum_{f\in\Ical_{F,\ell}}\int_{\hatK}\hat{\xvec}\,({\bbB_\ell}^{-1}\,\vvec_{F_f})\cdot\grad_{\hat{\xvec}}\hat{\phi}_{F_f}\\
\\
&&\ds-\int_{\hatK}\hat{\xvec}\,\div_{\hat{\xvec}}({\bbB_\ell}^{-1}\hat{\vvec}_\ell).
\end{array}
\end{equation}
Using integration formula \eqref{eq:coorbary}, and Cauchy-Schwarz inequality to bound the last term of \eqref{eq:vKl}, we have:
\begin{equation}
\label{eq:estim3}
|\vvec_{K_\ell}|^2\lesssim\sigma^2\,
\left(\sum_{i\in\Ical_{S,\ell}}|\vvec_{S_i}|^2+\sum_{f\in\Ical_{F,\ell}}|\vvec_{F_f}|^2+\norm{\Grad_{\hat{\xvec}}\hat{\vvec}_\ell}_{\bbL^2(\hatK)}^2\right).
\end{equation}
\begin{prop}\label{pro:rh-FS}
The Fortin-Soulie interpolation operator $\Pi_{FS}$ is such that:
\begin{equation}\label{eq:rh-FS}
\exists C_{FS}>0,\,\forall\vvec\in\bH^1(\Om),\quad\|\Pi_{FS}\vvec\|_h\leq C_{FS}\|\Grad\vvec\|_{\bbL^2(\Om)}.
\end{equation}
\end{prop}
\begin{proof}
Let $\vvec\in\bH^1(\Om)$. Let us set $\ul{\vvec}\in\Pcal_h\bH^1$ such that $\forall\ell\in\Ical_K$, $\ul{\vvec}_{|K_\ell}:=\vvec-\int_{K_\ell}\vvec/|K_\ell|$. We have:
\[
\norm{\Pi_{FS}\vvec}_h^2=\sum_{\ell\in\Ical_K}\,\norm{\Grad\Pi_{FS}\vvec}_{\bbL^2(K_\ell)}^2=\sum_{\ell\in\Ical_K}\,\norm{\Grad\Pi_{FS}\ul{\vvec}}_{\bbL^2(K_\ell)}^2.
\]
For $\avec$, $\bvec\in\R^2$, we set: $\avec\otimes\bvec:=(a_i\,b_j)_{i,j=1}^2\in\R^{2\times 2}$. According to equation \eqref{eq:PiFS(v)}, we have: 
\[
\Grad\Pi_{FS}\ul{\vvec}=\sum_{i\in\Ical_{S,\ell}}\ul{\vvec}_{S_i}\otimes\grad\psi_{S_i}+\frac{3}{2}\sum_{f\in\Ical_{F,\ell}}\ul{\vvec}_{F_f}\otimes\grad\phi_{F_f}+\ul{\vvec}_{K_\ell}\otimes\grad\phi_{K_\ell}.
\]
We can then make the estimate:
\[
\begin{array}{rcl}
\|\Grad\Pi_{FS}\ul{\vvec}\|_{\bbL^2(K_\ell)}^2&\lesssim&\ds \sum_{i\in\Ical_{S,\ell}}|\ul{\vvec}_{S_i}|^2\,
\|\grad\psi_{S_i}\|_{\bbL^2(K_\ell)}^2\\
\\
&&\ds+\sum_{f\in\Ical_{F,\ell}}|\tilde{\ul{\vvec}}_{F_f}|^2\,
\|\grad\phi_{F_f}\|_{\bbL^2(K_\ell)}^2\\
\\
&&\ds+\,|\ul{\vvec}_{K_\ell}|^2\,\|\grad\phi_{K_\ell}\|_{\bbL^2(K_\ell)}^2,\\
&\lesssim&\ds\norm{{\bbB_\ell}^{-1}}^2\,|J_\ell|\,\left(\,\sum_{i\in\Ical_{S,\ell}}|\ul{\vvec}_{S_i}|^2+\sum_{f\in\Ical_{F,\ell}}|\ul{\vvec}_{F_f}|^2+|\ul{\vvec}_{K_\ell}|^2\right).
\end{array}
\]
Thus, using estimate \eqref{eq:estim3}, and noticing that $\norm{{\bbB_\ell}^{-1}}^2\,|J_\ell|\lesssim\sigma^2$, we have:
\begin{equation}\label{eq:estim31}
\begin{array}{l}
\|\Grad\Pi_{FS}\ul{\vvec}\|_{\bbL^2(K_\ell)}^2\\
\quad\ds\lesssim\sigma^4\,
\left(\,\sum_{i\in\Ical_{S,\ell}}|\ul{\vvec}_{S_i}|^2+\sum_{f\in\Ical_{F,\ell}}|\ul{\vvec}_{F_f}|^2+\|\Grad_{\hat{\xvec}}\hat{\vvec}_\ell\|_{\bbL^2(\hatK)}^2\right).
\end{array}
\end{equation}
Using inequality \eqref{eq:estim1} and Poincaré-Steklov inequality \eqref{eq:PSinC} in cell, component by component, we have:
\begin{equation}\label{eq:estim12}
|\ul{\vvec}_{S_i}|^2\lesssim \|\Grad_{\hat{\xvec}}\hat{\vvec}_{\ell_i}\|_{\bbL^2(\hatK)}^2.
\end{equation}
Since the triangulation $\Tcal_h$ is suppose to be shape-regular, there exists  a constant $N_\theta\lesssim\sigma^2$ such that for all $i\in \Ical_{K,i}$, $N_i\leq N_\theta$ \cite[Rmk. 11.5]{ErGu21-I}. We then obtain that:
\[
\sum_{\ell\in\Ical_K}\sum_{i\in\Ical_{S,\ell}}\,\|\Grad_{\hat{\xvec}}\hat{\vvec}_{\ell_i}\|_{\bbL^2(\hatK)}^2\lesssim\,N_\theta\,\sum_{\ell\in\Ical_K}\|\Grad_{\hat{\xvec}}\hat{\vvec}_{\ell}\|_{\bbL^2(\hatK)}^2.
\]
Thus, summing \eqref{eq:estim12}, we get:
\begin{equation}\label{eq:estim13}
\sum_{\ell\in\Ical_K}\,\sum_{i\in\Ical_{S,\ell}}\,|\ul{\vvec}_{S_i}|^2\lesssim\,N_\theta\,\sum_{\ell\in\Ical_K}\,\|\Grad_{\hat{\xvec}}\hat{\vvec}_\ell\|_{\bbL^2(\hatK)}^2.
\end{equation} 
Using the triangular inequality, equality \eqref{eq:vKvF}, and inequality \eqref{eq:MTI-PSinC-ref}, we obtain:
\[
|\ul{\vvec}_{F_f}|^2\lesssim|F_f|^{-1}\,\sum_{\ell'\in\Ical_{K,f}}\|\ul{\vvec}_{\ell'|F_f}\|_{\bL^2(F_f)}^2\,\lesssim\sum_{\ell'\in\Ical_{K,f}}\|\Grad_{\hat{\xvec}}\hat{\vvec}_{\ell'}\|_{\bbL^2(\hatK)}^2.
\]
Notice that: $\ds\sum_{\ell\in\Ical_K}\,\sum_{f\in\Ical_{F,\ell}}\sum_{\ell'\in\Ical_{K,f}}\|\Grad_{\hat{\xvec}}\hat{\vvec}_{\ell'}\|_{\bbL^2(\hatK)}^2\,\leq\,6\,\sum_{\ell\in\Ical_K}\|\Grad_{\hat{\xvec}}\hat{\vvec}_\ell\|_{\bbL^2(\hatK)}^2$, so that:
\begin{equation}\label{eq:estim22}
\sum_{\ell\in\Ical_K}\,\sum_{f\in\Ical_{F,\ell}}|F_f|^{-1}\,\|\ul{\vvec}_{\ell'|F_f}\|_{\bL^2(F_f)}^2\lesssim\sum_{\ell\in\Ical_K}\|\Grad_{\hat{\xvec}}\hat{\vvec}_\ell\|_{\bbL^2(\hatK)}^2.
\end{equation}
Summing \eqref{eq:estim31} over $\ell\in\Ical_K$, using \eqref{eq:estim13} and \eqref{eq:estim22}, we get that:
\begin{equation}\label{eq:estim32}
\|\Pi_{FS}\vvec\|_h^2
\lesssim
\sigma^2\,N_\theta\,\sum_{\ell\in\Ical_K}\,\|\Grad_{\hat{\xvec}}\hat{\vvec}_\ell\|_{\bbL^2(\hatK)}^2,
\end{equation}
Considering \eqref{eq:gradvK} for each component, and noticing that $\norm{{\bbB_\ell}}^2\,|J_\ell|^{-1}\lesssim\sigma^2$, we obtain that:
\begin{equation}\label{eq:estim33}
\|\Pi_{FS}\vvec\|_h^2
\lesssim\sigma^4\,N_\theta\,\sum_{\ell\in\Ical_K}\,\norm{\Grad\vvec_\ell}_{\bbL^2(K_\ell)}^2.
\end{equation}
We obtain then \eqref{eq:rh-FS} with $\ds C_{FS}\approx\sigma^2\,(N_\theta)^{1/2}$. 
\end{proof}
We recall that the discrete Poincaré–Steklov inequality \eqref{eq:cps-constant} holds.
\begin{theo}
Let $\Xcal_h=\Xcal_{FS}$. Then the continuous bilinear form $a_{S,h}(\cdot,\cdot)$ is $T_h$-coercive and Problem \eqref{eq:Stokes-FVah} is well-posed.
\end{theo}
\begin{proof}
Apply the proof of Theorem \ref{thm:WellPosed-nc}.
\end{proof}
Notice that in the recent paper \cite{SaTo22}, the inf-sup condition of the mixed Fortin-Soulie finite element is proven directly on a triangle and then using the macro-element technique \cite{Sten90}, but it seems difficult to use this technique to build a Fortin operator, which is needed to compute error estimates.\\
The study can be extended to higher orders for $d=2$ using the following papers: \cite{BaSt07} for $k\geq 4$, $k$ even, \cite{CaSa21} for $k=3$ and \cite{CaSa22}  for $k\geq 5$, $k$ odd. In \cite{DiST22}, the authors propose a local Fortin operator for the lowest order Taylor-Hood finite element \cite{TaHo73} for $d=3$ which could be used to prove the T-coercivity.
\section{Numerical results}\label{sec:resu}
Consider Problem \eqref{eq:Stokes} with data $\fvec=-\grad\phi$, where $\phi\in H^1(\Om)\cap L^2_{zmv}(\Om)$. The unique solution is then $(\uvec,p):=(0,\phi)$. By integrating by parts, the source term in \eqref{eq:Stokes-FV} reads: 
\begin{equation}
\label{eq:contIPP}
\forall\vvec\in\bH^1_0(\Om),\quad\int_\Om\fvec\cdot\vvec=\int_\Om\phi\,\div\vvec.
\end{equation}
Recall that the nonconforming space $\bX_h$ defined in \eqref{eq:Xh} is a subset of $\Pcal_h\bH^1$: using a nonconforming finite element method, the integration by parts must be done on each element of the triangulation, and we have:
\begin{equation}
\label{eq:discIPP}
\forall\vvec\in\Pcal_h\bH^1,\quad\int_\Om\fvec\cdot\vvec=(\div_h\vvec,\phi)+\sum_{f\in\Ical_F}\int_{F_f}[\vvec\cdot\nvec_f]\,\phi.
\end{equation}
When we apply this result to the right-hand-side of \eqref{eq:Stokes-FVah}, we notice that the term with the jumps acts as a numerical source, which numerical influence is proportional to $1/\nu$. Thus, we cannot obtain exactly $\uvec_h=0$ (see also \eqref{eq:aprioriEE}). Linke proposed in \cite{Link14} to project the test function $\vvec_h\in\bX_h$ on a discrete subspace of $\bH(\div;\,\Om)$, like Raviart-Thomas or Brezzi-Douglas-Marini finite elements (see \cite{RaTh77,BrDM85}, or the monograph \cite{BoBF13}). Let $\Pi_{\div}:\bX_{0,h}\rightarrow P^k_{disc}(\Tcal_h)\cap\bH_0(\div;\Om)$ be some interpolation operator built so that for all $\vvec_h\in\bX_{0,h}$, for all $\ell\in\Ical_K$, $(\div\Pi_{\div}\vvec_h)_{|K_\ell}=\div\vvec_{h|K_\ell}$. Integrating by parts, we have for all $\vvec_h\in\bX_{0,h}$:
\[
\begin{array}{rcl}
\ds\int_\Om\fvec\cdot\Pi_{\div}\vvec_h&=&\ds\int_\Om\phi\,\div\Pi_{\div}\vvec_h=\sum_{\ell\in K_\ell}\int_{K_\ell}\phi\,\div\Pi_{\div}\vvec_h,\\
&=&\ds\sum_{\ell\in K_\ell}\int_{K_\ell}\phi\,\div\vvec_h=(\div_h\vvec_h,\phi).
\end{array}
\]
The projection $\Pi_{\div}$ allows to eliminate the terms of the integrals of the jumps in \eqref{eq:discIPP}. \\
Let us write Problem \eqref{eq:Stokes-FVah} as:\\
Find $(\uvec_h,p_h)\in\Xcal_h$ such that
\begin{equation}
\label{eq:Stokes-FVah-RT}
a_{S,h}\,\left((\uvec_h,p_h),(\vvec_h,q_h)\right)=\ell_\fvec\,\left(\,(\Pi_{\div}\vvec_h,q_h)\,\right)\quad\forall(\vvec_h,q_h)\in\Xcal_h.
\end{equation}
In the case of $\Xcal_h=\Xcal_{CR}$ and a projection on Brezzi-Douglas-Marini finite elements, the following error estimate holds if $(\uvec,p)\in\bH^2(\Om)\times H^1(\Om)$: 
\begin{equation}
\label{eq:aprioriEE-PiBDM}
\|\uvec-\uvec_h\|_{\bL^2(\Om)}\leq \widetilde{C}\,h^2\,|\uvec|_{\bH^2(\Om)},
\end{equation}
where the constant $\widetilde{C}$ if independent of $h$. The proof is detailed in \cite{BLMS15} for shape-regular meshes and \cite{AKLM22} for anisotropic meshes. We remark that the error doesn't depend on the norm of the pressure nor on the $\nu$ parameter.  We will provide some numerical results to illustrate the effectiveness of this formulation, even with a projection on the Raviart-Thomas finite elements, which, for a fixed polynomial order, are less precise than the Brezzi-Douglas-Marini finite elements.\\
For all $\ell\in\Ical_K$, we let $P^k_H(K_\ell)$ be the set of homogeneous polynomials of order $k$ on $K_\ell$.
\\ For $k\in\N^\star$, the space of Raviart-Thomas finite elements can be defined as:
\[
\bX_{RT_k}:=\left\{\vvec\in\bH(\div;\,\Om);\,\forall \ell\in\Ical_k,\,\vvec_{|K_\ell}=\avec_\ell+b_\ell\xvec\,|\,(\avec_\ell,b_\ell)\in P^k(K_\ell)^d\times P^k_H(K_\ell)\right\}.
\]
Let $k\leq 1$. \\
The Raviart–Thomas interpolation operator $\Pi_{RT_k}:\bH^1(\Om)\cup\bX_h\rightarrow\bX_{RT_k}$ is defined by: $\forall\vvec\in\bH^1(\Om)\cup\bX_h$,
\begin{equation}
\label{eq:PiRTdef}
\left\{
\begin{array}{rl}
\forall f\in\Ical_F,&\ds\int_{F_f}\Pi_{RT_k}\vvec\cdot\nvec_f\,q=\int_{F_f}\vvec\cdot\nvec_f\,q,\quad\forall q\in P^k(F_f)\\
\mbox{for }k=1,\,\forall\ell\in\Ical_K,&\ds\int_{K_\ell}\Pi_{RT_1}\vvec=\int_{K_\ell}\vvec
\end{array}\right..
\end{equation}
Note that the Raviart–Thomas interpolation operator preserves the constants. Let $\vvec_h\in\bX_h$. In order to compute  the left-hand-side of \eqref{eq:discIPP}, we must evaluate $(\Pi_{RT_k}\vvec_h)_{|K_\ell}$ for all $\ell\in\Ical_K$. Calculations are performed using the proposition below, which corresponds to \cite[Lemma 3.11]{Gati14}:
\begin{prop}\label{prop:PiRTref}
Let $k\leq 1$. Let $\hat{\Pi}_{RT_k}:\bH^1(\hatK)\rightarrow\bP^k(\hatK)$ be the Raviart–Thomas interpolation operator restricted to the reference element, so that: $\forall\hat{\vvec}\in\bH^1(\hatK)$,
\begin{equation}
\label{eq:PiRTref}
\left\{\begin{array}{rl}
\forall \hatF\in\pa\hatK,&
\ds\int_{\hatF}\hat{\Pi}_{RT_k}\hat{\vvec}\cdot\nvec_{\hatF}\,\hat{q}=\int_{\hatF}\hat{\vvec}\cdot\nvec_{\hatF}\,\hat{q},\quad\forall \hat{q}\in P^k(\hatF)\\
\mbox{\emph{for} }k=1,&\ds\int_{\hatK}\hat{\Pi}_{RT_k}\hat{\vvec}=\int_{\hatK}\hat{\vvec}
\end{array}\right..
\end{equation}
We then have: $\forall\ell\in\Ical_K$,
\begin{equation}
\label{eq:PiRTrefPiRT}
(\Pi_{RT_k}\vvec)_{|K_\ell}(\xvec)=
\bbB_\ell\,\left(\,\hat{\Pi}_{RT_k}\bbB_\ell{}^{-1}\hat{\vvec}_\ell\,\right)\circ T_\ell{}^{-1}(\xvec)\,\mbox{ where }\,\hat{\vvec}_\ell=\vvec\circ T_\ell(\hat{\xvec}).
\end{equation} 
\end{prop}
The proof is based on the equality of the $\hatF$ and $\hatK$-moments of $(\Pi_{RT_k}\vvec)_{|K_\ell}\circ T_\ell(\hat{\xvec})$ and $\bbB_\ell\left(\,\hat{\Pi}_{RT_k}\bbB_\ell{}^{-1}\hat{\vvec}_\ell\,\right)(\hat{\xvec})$. For $k=0$, setting for $d'\in\{1,\cdots,d\}$: ${\mbox{\boldmath{$\psi$}}}_{f,{d'}}:=\psi_f\,\evec_{d'}$, we obtain that:
\begin{equation}
\label{eq:PiRTCR}
\forall\ell\in\Ical_K\,,\,\forall f\in\Ical_{F,\ell}\,,\quad(\Pi_{RT_0}{\mbox{\boldmath{$\psi$}}}_{f,{d'}})_{|K_\ell}=(d\,|K_\ell|)^{-1}\,\left(\xvec-\vec{OS}_{f,\ell}\right)\Scal_{f,\ell}\cdot\evec_{d'}\,,
\end{equation}
where $S_{f,\ell}$ is the vertex opposite to $F_f$ in $K_\ell$.\\
For $k=1$, the vector $(\Pi_{RT_1}\vvec_h)_{|K_\ell}$ is described by eight unknowns:
\[
(\Pi_{RT_1}\vvec_h)_{|K_\ell}=\bbA_\ell\,\xvec+(\bvec_\ell\cdot\xvec)\,\xvec+\dvec_\ell,\mbox{ where }\bbA_\ell\in\R^{2\times2},\,\bvec_\ell\in\R^2,\,\dvec_\ell\in\R^2.
\]
We compute only once the inverse of the matrix of the linear system \eqref{eq:PiRTref}, in $\R^{8\times 8}$.
\\
In the Table \ref{tab:grad} (resp. Tables \ref{tab:cossin-3} and \ref{tab:cossin-4}), we call $\eps_0(\uvec)=\|\uvec_h\|_{\bL^2(\Om)}$ (resp. $\|\uvec-\uvec_h\|_{\bL^2(\Om)}/\|\uvec\|_{\bL^2(\Om)}$) the velocity error in $\bL^2(\Om)$-norm, where $\uvec_h$ is the solution to Problem \eqref{eq:Stokes-FVah} (columns $\bX_{CR}$ and $\bX_{FS}$) or \eqref{eq:Stokes-FVah-RT} (columns $\bX_{CR}+\Pi_{RT_0}$ and $\bX_{FS}+\Pi_{RT_1}$) and $h$ is the mesh step.
\\
We first consider Stokes Problem \eqref{eq:Stokes} in $\Omega=(0,1)^2$ with $\uvec=0$, $p=(x_1)^3+(x_2)^3-0.5$, $\fvec=\grad p=3\,\left(\,(x_1)^2,(x_2)^2\,\right)^T$. We report in Table \ref{tab:grad} $\eps_0(\uvec)$ for $h=5.00\,e-2$ and for different values of $\nu$.
\begin{table}[h]
\begin{tabular}{ccccc}
\hline
$\nu$&$\bX_{CR}$&$\bX_{CR}+\Pi_{RT_0}$&$\bX_{FS}$&$\bX_{FS}+\Pi_{RT_1}$\\
\hline
$1.00\,e+0$ & $3.19\,e-4$ & $1.34\,e-18$ & $3.53\,e-7$ & $9.09\,e-19$\\
$1.00\,e-3$ & $3.19\,e-1$ & $1.34\,e-15$ & $3.53\,e-4$ & $9.09\,e-16$\\
$1.00\,e-4$ & $3.19\,e+0$ & $1.34\,e-14$ & $3.53\,e-3$ & $9.09\,e-15$\\
\hline
\end{tabular}
\caption{\label{tab:grad}Values of $\eps_0(\uvec)$ for $h=5.00\,e-2$}
\end{table}
\\%
When there is no projection, the error is inversely proportional to the $\nu$ parameter, whereas using the projection, we obtain $\uvec_h=0$ up to machine precision. 
\\
We now consider Stokes Problem \eqref{eq:Stokes} in $\Omega=(0,1)^2$ with:
\[
\uvec=\left(\begin{matrix}(1-\cos(2\,\pi\,x_1))\,\sin(2\,\pi\,x_2)\\
(\cos(2\,\pi\,x_2)-1)\,\sin(2\,\pi\,x_1)\end{matrix}\right),\quad \left\{\begin{array}{rcl}
p&=&\sin(2\,\pi\,x_1)\,\sin(2\,\pi\,x_2),\\
\fvec&=&-\nu\,\Delta\uvec+\grad p.\end{array}\right.
\]
We report in Table \ref{tab:cossin-3} (resp. \ref{tab:cossin-4}) the values of $\eps_0(\uvec)$ in the case $\nu=1.00\,e-3$ (resp. $\nu=1.00\,e-4$) for different level of mesh refinement. When there is no projection, $\eps_0(\uvec)$ is inversely proportional to $\nu$, whereas using the projection, $\eps_0(\uvec)$ is independent of $\nu$.

\begin{table}[h]
\begin{tabular}{ccccc}
\hline
$h$&$\bX_{CR}$&$\bX_{CR}+\Pi_{RT_0}$&$\bX_{FS}$&$\bX_{FS}+\Pi_{RT_1}$\\
\hline
$5.00\,e-2$ & $5.66\,e-1$ & $1.13\,e-2$ & $2.35\,e-3$ & $2.06\,e-4$\\
$2.50\,e-2$ & $1.33\,e-1$ & $2.89\,e-3$ & $3.21\,e-4$ & $2.59\,e-5$\\
$1.25\,e-2$ & $3.88\,e-2$ & $5.40\,e-4$ & $4.20\,e-5$ & $3.40\,e-6$\\
$6.25\,e-3$ & $8.40\,e-3$ & $1.79\,e-4$ & $5.04\,e-6$ & $4.15\,e-7$\\
Rate        & $h^{2.05}$  & $h^{2.07}$  & $h^{2.96}$  & $h^{2.98}$\\
\hline
\end{tabular}
\caption{\label{tab:cossin-3}Values of $\eps_0(\uvec)$ for $\nu=1.00\,e-3$}
\end{table}%
\begin{table}[h]
\begin{tabular}{ccccc}
\hline
$h$&$\bX_{CR}$&$\bX_{CR}+\Pi_{RT_0}$&$\bX_{FS}$&$\bX_{FS}+\Pi_{RT_1}$\\
\hline
$5.00\,e-2$ & $5.66\,e-0$ & $1.13\,e-2$ & $2.35\,e-2$ & $2.06\,e-4$\\
$2.50\,e-2$ & $1.33\,e-0$ & $2.89\,e-3$ & $3.20\,e-3$ & $2.59\,e-5$\\
$1.25\,e-2$ & $3.38\,e-1$ & $5.40\,e-4$ & $4.20\,e-4$ & $3.40\,e-6$\\
$6.25\,e-3$ & $8.40\,e-2$ & $1.79\,e-4$ & $5.04\,e-5$ & $4.15\,e-7$\\
Rate        & $h^{2.05}$  & $h^{2.07}$  & $h^{2.96}$  & $h^{2.98}$\\
\hline
\end{tabular}
\caption{\label{tab:cossin-4}Values of $\eps_0(\uvec)$ for $\nu=1.00\,e-4$}
\end{table}
%
Let $\uvec_{FS}$ (resp. $\uvec_{FS+RT_1}$) the solution to Problem \eqref{eq:Stokes-FVah} (resp. \eqref{eq:Stokes-FVah-RT}) computed with Fortin-Soulie finite elements. We represent on Figure \ref{fig:dU} the values of the Lagrange projection of $(\uvec_{FS}-\uvec_{FS+RT_1})$ in the case where $h=2.50\,e-2$ and $\nu=1.00\,e-4$.  We observe local oscillations, of order the mesh step, which are caused by the numerical source exhibited in \eqref{eq:discIPP}.
 \begin{figure}[t]
 \includegraphics[width=1\linewidth]{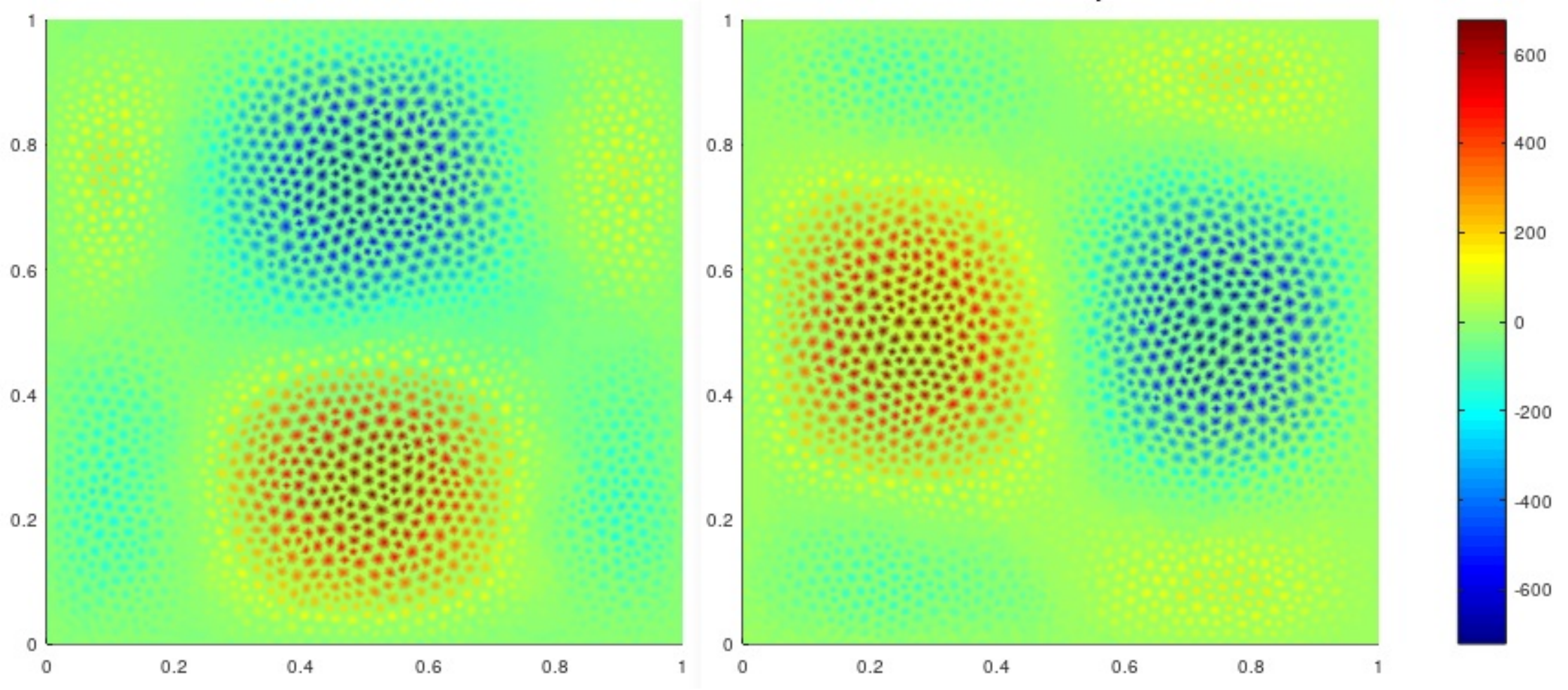}
 \caption{Values of $(\uvec_{FS}-\uvec_{FS+RT_1})$. Left: $x_1$-component, right: $x_2$-component.}
 \label{fig:dU}
 \end{figure}
\\ 
In order to enhance the numerical results, one can also use a posteriori error estimators to adapt the mesh (see \cite{DaDP95,DoAi05} for order $1$ and \cite{AABR12} for order $2$).
\\
Alternatively, using the nonconforming Crouzeix-Raviart mixed finite element method, one can build a divergence-free basis, as described in \cite{Hech81}. Notice that using conforming finite elements, the Scott-Vogelius finite elements \cite{ScVo85,Zhan05} produce velocity approximations that are exactly divergence free.
\\
The code used to get the numerical results can be downloaded on GitHub \cite{StokesNCFEM}. 
\section{Conclusion}
We analysed the discretization of Stokes problem with nonconforming finite elements in light of the T-coercivity theory, we computed stability coefficients for $k=1$, $d=2$ or $3$ without regularity assumption; and for $k=2$, $d=2$ in the case of a shape-regular simplicial triangulation sequence. For $k=2$, we used an alternative definition of the Fortin-Soulie interpolation operator. We then provided numerical results to illustrate the importance of using $\bH(\div)$-conforming projection. Further, we intend to extend the study to other mixed finite element methods.
\section*{Acknowledgements}
The author acknowledges Mahran Rihani and Albéric Lefort.
\bibliographystyle{unsrt}
\bibliography{Jame22-TC-Stokes}

\begin{thebibliography}{10}

\bibitem{Ciar12}
P.~{Ciarlet~Jr.}
\newblock {T-coercivity: Application to the discretization of Helmhotz-like
  problems}.
\newblock {\em Computers \& Mathematics with Applications}, 64(1):22--24, 2012.

\bibitem{JaCi13}
E.~Jamelot and {P. Ciarlet, Jr.}
\newblock {Fast non-overlapping Schwarz domain decomposition methods for
  solving the neutron diffusion equation}.
\newblock {\em Journal of Computational Physics}, 241:445--463, 2013.

\bibitem{CiJK17}
P.~{Ciarlet~Jr.}, E.~Jamelot, and F.~D. Kpadonou.
\newblock {Domain decomposition methods for the diffusion equation with
  low-regularity solution}.
\newblock {\em Computers \& Mathematics with Applications}, 74(10):2369--2384,
  2017.

\bibitem{Gire18}
L.~Giret.
\newblock {\em {Non-Conforming Domain Decomposition for the Multigroup Neutron
  SPN Equation}}.
\newblock PhD thesis, {Universit\'e Paris-Saclay}, 2018.

\bibitem{GiRa86}
V.~Girault and P.-A. Raviart.
\newblock {\em {Finite element methods for Navier-Stokes equations}}.
\newblock Springer-Verlag, 1986.

\bibitem{CrRa73}
M.~Crouzeix and P.-A. Raviart.
\newblock {Conforming and nonconforming finite element methods for solving the
  stationary Stokes equations}.
\newblock {\em RAIRO, S\'er. Anal. Num\'er.}, 7(3):33--75, 1973.

\bibitem{FoSo83}
M.~Fortin and M.~Soulie.
\newblock {A non-conforming piecewise quadratic finite element on triangles}.
\newblock {\em International Journal for Numerical Methods in Engineering},
  19(4):505--520, 1983.

\bibitem{ErGu21-II}
A.~Ern and J.-L. Guermond.
\newblock {\em {{Finite elements II}}}, volume~73 of {\em Texts in Applied
  Mathematics}.
\newblock Springer, 2021.

\bibitem{BCDG16}
C.~Bernardi, M.~Costabel, M.~Dauge, and V.~Girault.
\newblock {Continuity properties of the inf-sup constant for the divergence}.
\newblock {\em SIAM J. Math. Anal.}, 48(2):1250--1271, 2016.

\bibitem{Gall19}
D.~Gallistl.
\newblock {Rayleigh-Ritz approximation of the inf-sup constant for the
  divergence}.
\newblock {\em Mathematics of Computation}, 88(315):73--89, 2019.

\bibitem{CoDa15}
M.~Costabel and M.~Dauge.
\newblock {On the inequalities of Babu\v{s}ka-Aziz, Friedrichs and
  Horgan-Payne}.
\newblock {\em Arch. Rational Mech. Anal.}, 217:873--898, 2015.

\bibitem{BaCi22}
M.~Barr{\'e} and P.~Ciarlet Jr.
\newblock {T-coercivit{\'e} et probl{\`e}mes mixtes}.
\newblock working paper or preprint, October 2022.

\bibitem{Ciar21}
P.~{Ciarlet~Jr.}
\newblock {M\'ethodes variationnelles pour l'analyse de probl\`emes non
  coercifs}, 2021.
\newblock M.Sc. AMS lecture notes (ENSTA-IPP).

\bibitem{TaHo73}
C.~Taylor and T.~Hood.
\newblock {Numerical solution of the Navier-Stokes equations using the finite
  element technique}.
\newblock {\em Computers \& Fluids}, 1:73--100, 1973.

\bibitem{BoBF13}
D.~Boffi, F.~Brezzi, and M.~Fortin.
\newblock {\em {Mixed and hybrid finite element methods and applications}}.
\newblock Springer-Verlag, 2013.

\bibitem{ErGu21-I}
A.~Ern and J.-L. Guermond.
\newblock {\em {{Finite elements I}}}, volume~72 of {\em Texts in Applied
  Mathematics}.
\newblock Springer, 2021.

\bibitem{Ciar91}
P.~G. Ciarlet.
\newblock {The Effect of Numerical integration for Second-Order Problems}.
\newblock In {\em Finite Element Methods (Part 1)}, volume~II of {\em Handbook
  of Numerical Analysis}. Elsevier, 1991.

\bibitem{Saut22}
S.~Sauter.
\newblock {The inf-sup constant for Crouzeix-Raviart triangular elements of any
  polynomial order}, 2022.

\bibitem{ApNS01}
T.~Apel, S.~Nicaise, and J.~{Sch\"oberl}.
\newblock {Crouzeix-Raviart type finite elements on anisotropic meshes}.
\newblock {\em Numerische Mathematik}, 89(2):193--223, 2001.

\bibitem{CaSa22}
C.~Carstensen and S.~Sauter.
\newblock {Critical functions and inf-sup stability of Crouzeix-{R}aviart
  elements}.
\newblock {\em Computers and Mathematics with Applications}, 108:12--23, 2022.

\bibitem{ScZh90}
L.~R. Scott and S.~Zhang.
\newblock Finite element interpolation of nonsmooth functions satisfying
  boundary conditions.
\newblock {\em Math. Comp.}, 54:483--493, 1990.

\bibitem{Ciar13}
P.~{Ciarlet~Jr.}
\newblock {Analysis of the Scott–Zhang interpolation in the fractional order
  Sobolev spaces}.
\newblock {\em J. Numer. Math.}, 21(3):173–--180, 2013.

\bibitem{SaTo22}
S.~Sauter and C.~Torres.
\newblock {On the Inf-Sup Stabillity of Crouzeix-Raviart Stokes Elements in
  3D}, 2022.

\bibitem{Sten90}
R.~Stenberg.
\newblock Error analysis of some finite element methods for the stokes problem.
\newblock {\em Math. Comp.}, 54:495--508, 1990.

\bibitem{BaSt07}
{\`A}.~Baran and G.~Stoyan.
\newblock {Gauss-Legendre elements: a stable, higher order non-conforming
  finite element family}.
\newblock {\em Computing}, 79(1):1--21, 2007.

\bibitem{CaSa21}
C.~Carstensen and S.~Sauter.
\newblock {Crouzeix-{R}aviart triangular elements are inf-sup stable}, 2021.
\newblock Preprint.

\bibitem{DiST22}
L.~Diening, J.~Storn, and T.~Tscherpel.
\newblock {Fortin operator for Taylor-Hood element}.
\newblock {\em Numerische Mathematik}, 150(2):671--689, 2022.

\bibitem{Link14}
A.~Linke.
\newblock {On the Role of the Helmholtz-Decomposition in Mixed Methods for
  Incompressible Flows and a New Variational Crime}.
\newblock {\em Comput. Methods Appl. Mech. Engrg.}, 268:782--800, 2014.

\bibitem{RaTh77}
P.-A. Raviart and J.-M. Thomas.
\newblock {A mixed finite element method for second order elliptic problems}.
\newblock In {\em Mathematical aspects of finite element methods}, volume 606
  of {\em {Lecture Notes in Mathematics}}, pages 292--315. Springer, 1977.

\bibitem{BrDM85}
F.~Brezzi, J.~Douglas, and L.~D. Marini.
\newblock {Two families of mixed finite elements for second order elliptic
  problems}.
\newblock {\em Numerische Mathematik}, 47(2):217--235, 1985.

\bibitem{BLMS15}
C.~Brennecke, A.~Linke, C.~Merdon, and J.~{Sch\"oberl}.
\newblock {Optimal and pressure independent $L^2$ velocity error estimates for
  a modified Crouzeix-Raviart Stokes element with BDM reconstructions}.
\newblock {\em Journal of Computational Mathematics}, 33(2):191--208, 2015.

\bibitem{AKLM22}
T.~Apel, V.~Kempf, A.~Linke, and C.~Merdon.
\newblock {A nonconforming pressure-robust finite element method for the Stokes
  equations on anisotropic meshes}.
\newblock {\em IMA Journal of Numerical Analysis}, 42(1):392--416, 2022.

\bibitem{Gati14}
G.~N. Gatica.
\newblock {\em {A Simple Introduction to the Mixed Finite Element Method:
  Theory and Applications}}.
\newblock SpringerBriefs in Mathematics. Springer, 2014.

\bibitem{DaDP95}
E.~Dari, R.~Dur{\'a}n, and C.~Padra.
\newblock {Error estimators for nonconforming finite element approximations of
  the Stokes problem}.
\newblock {\em Mathematics of Computation}, 64(211):1017--1033, 1995.

\bibitem{DoAi05}
W.~{D\"orfler} and M.~Ainsworth.
\newblock {Reliable a posteriori error control for nonconforming finite element
  approximation of Stokes flow}.
\newblock {\em Mathematics of Computation}, 74(252):1599--1619, 2005.

\bibitem{AABR12}
M.~Ainsworth, A.~Allendes, G.~R. Barrenechea, and R.~Rankin.
\newblock {Computable error bounds for nonconforming Fortin–Soulie finite
  element approximation of the Stokes problem}.
\newblock {\em IMA Journal of Numerical Analysis}, 32(2):417--447, 2011.

\bibitem{Hech81}
F.~Hecht.
\newblock {Construction d'une base de fonctions $P_1$ non conforme {à}
  divergence nulle dans $\R^3$}.
\newblock {\em {RAIRO, S\'er. Anal. Num\'er.}}, 15(2):119--150, 1981.

\bibitem{ScVo85}
{L. R. Scott and M. Vogelius}.
\newblock {Norm estimates for a maximal right inverse of the divergence
  operator in spaces of piecewise polynomials}.
\newblock {\em RAIRO, S\'er. Anal. Num\'er.}, 19(1):111--143, 1985.

\bibitem{Zhan05}
{S. Zhang}.
\newblock {A new family of stable mixed finite elements for the 3D Stokes
  equations}.
\newblock {\em Mathematics of Computation}, 74:543--554, 2005.

\bibitem{StokesNCFEM}
{E. Jamelot}.
\newblock \url{https://github.com/cea-trust-platform/Stokes_NCFEM}.

\end{thebibliography}

\end{document}